\documentclass[preprint,12pt]{elsarticle}
\usepackage{bbm}
\usepackage{amsmath,amssymb,mathrsfs,stmaryrd,amsthm}
\usepackage{mathtools,bm}
\usepackage{eufrak,lmodern}
\usepackage[ruled,linesnumbered]{algorithm2e}
\newcommand{\kuo}[1]{\left( #1 \right)}
\newcommand{\zkuo}[1]{\left[ #1 \right]}
\newcommand{\dkuo}[1]{\left \{ #1 \right \}}
\newcommand{\e}{\mathbb{E}}
\newcommand{\p}{\mathbb{P}}
\usepackage{xcolor}
\newcommand{\abs}[1]{\left|#1\right|}

\newcommand{\brackt}[1]{\left\langle #1 \right\rangle}
\newcommand{\law}[1]{\mathcal{L}_{#1}}
\numberwithin{equation}{section}

\usepackage{amssymb}
\usepackage{lineno}
\usepackage[top=3.2cm,bottom=2.2cm,left=2.5cm,right=2.5cm]{geometry}
\journal{Journal of Mathematical Analysis and Applications}
\theoremstyle{plain}
\newtheorem{theorem}{Theorem}[section]

\newtheorem{lemma}[theorem]{Lemma}
\newtheorem{proposition}[theorem]{Proposition}
\theoremstyle{definition}
\newtheorem{definition}[theorem]{Definition}
\newtheorem{remark}[theorem]{Remark}
\newtheorem{assumption}{Assumption}[section]
\numberwithin{equation}{section}

\begin{document}

\begin{frontmatter}

\title{Mean reflected McKean-Vlasov stochastic differential equation}


\author[1]{Shaopeng Hong\corref{cor1}}
\ead{Hong_sp@outlook.com}
\author[2]{Sheng Xiao}
\address[1]{School of Statistics, Southwestern University of Finance and Economics, Chengdu
610074, China}
\address[2]{{School of Mathematics and Statistics, Hunan First Normal University, Changsha 410205, China}}
\cortext[cor1]{Corresponding author}
\begin{abstract}
In this paper, we investigate a class of mean reflected McKean-Vlasov stochastic differential equation (MR-MVSDE), which extends the equation proposed by \cite{briand2020particles} by allowing the solution's distribution to not only constrain its behavior but also affect the diffusion and drift coefficients.
We establish the existence and uniqueness results of MR-MVSDE, investigate the propagation of chaos, and examine the stability properties with respect to the initial condition, coefficients, and driving process. Moreover, we provide rigorous proof of a Freidlin-Wentzell type LDP using the weak convergence method. We also demonstrate the existence of an invariant measure and employ the coupling by change of measure method to prove log-Harnack inequality and shift Harnack inequality. 
\end{abstract}


\end{frontmatter}

\section{Introduction}
This paper considers the following \emph{mean reflected McKean-Vlasov stochastic differential equation} (MR-MVSDE for short).
\begin{equation}\label{eq:1.1}
\left\{\begin{array}{l}
X_t=\xi+\int_0^t b\left(X_s, \mathcal{L}_{X_s}\right) d s+\int_0^t \sigma\left(X_s,\mathcal{L}_{X_s}\right) d B_s+K_t, \quad t \geq 0 \\
\mathbb{E}\left[h\left(X_t\right)\right] \geq 0, \quad \int_0^t \mathbb{E}\left[h\left(X_s\right)\right] d K_s=0, \quad t \geq 0,
\end{array}\right.
\end{equation}
where $\mathcal{L}_{X_t}$ is the law of $X_t$, $(B_t)_{t\geq 0}$ is a standard Brownian motion defined on some complete probability space natural filtration $(\Omega,\mathcal{F},\{\mathcal{F}_t\}_{t\geq 0},\p)$. We will always assume that $h$ is nondecreasing and that the law of $\xi$ is such that $\e\zkuo{h(\xi)}\geq 0$. For the precise conditions on $b, \sigma, h$, we refer the reader to Section \ref{sec:2}. The solution to \eqref{eq:1.1} is a couple of continuous processes $(X, K)$, by the Skorokhod condition $\int_0^t \e[h(X_s)]dK_s = 0$, $K$ is required to ensure that the mean constraint $\e[h(X_t)]\geq 0 $ is satisfied in a minimal way.

The McKean-Vlasov stochastic differential equation (MV-SDE for short), also referred to as the mean-field stochastic differential equation or distribution-dependent stochastic differential equation, was first introduced in the literature by \cite{Mckean1966class} and \cite{vlasov1968vibrational}. In recent years, it has garnered significant attention from scholars. The MV-SDE is used to describe the limiting behavior of mean-field interactive particle systems as $n$ goes to infinity.
 We refer to \cite{sznitman1991topics} for a detailed exposition and references. 
 Currently, the MV-SDE plays a crucial role in various fields such as mean field games, financial derivatives pricing and measure value processes, see e.g. \citep{PTMFG, bush2011Stochastic, hambly2019SPDE, Dawson1995StochasticME} and references therein.

The reflected SDE was introduced by \cite{Skorokhod1961StochasticEF} and \cite{Skorokhod1962StochasticEF}. Since then, it has been widely researched, see e.g. \citep{Lions1984StochasticDE, lundstrom2019stochastic, burdzy2003heat, burdzy2004heat}. This type of equation restricts the path of the solution to a given domain. The reflected MV-SDE, first proposed by \cite{sznitman1984nonlinear}. 
Recently, \cite{adams2022large} studied the MV-SDE with superlinear growth self-stabilizing coefficients on unbounded convex regions, replacing the classical Wasserstein Lipschitz condition. 
Their work proved the existence and uniqueness of the solution, getting the result of the propagation of
 chaos, and examined the Freidlin-Wentzell type Large Deviations Principle (LDP for short), as well as Kramer's law for the exit-time from subdomains that lie within the interior of the reflecting domain. 
In a similar vein, \cite{wei2022mckean} investigated the MV-SDE with oblique reflection on non-smooth time-dependent domains. Their research proved the existence and uniqueness of the solution, demonstrated chaos propagation, and examined the Freidlin-Wentzell type LDP.


In their work, \cite{briand2018bsdes} introduced the concept of mean reflected backward stochastic differential equation (MR-BSDE for short), which places restrictions not on the path of the solution, but on its distribution. 
They proved that, under certain conditions, MR-BSDE has a unique solution. Further research on MR-BSDE can be found in \citep{briand2021particles, liu2019bsdes, cui2023well}, and their respective references. Inspired by MR-BSDE,
\cite{briand2020particles} studied the mean reflected stochastic differential equation (MR-SDE for short), established the well-posedness of the solution, and analyzed the propagation of chaos. They also proposed a numerical simulation method based on particle approximation techniques. \cite{briandMEAN} extended this work to L\'evy driven MR-SDE. \cite{falkowski2021mean} examined the MR-SDE with two reflective walls and solved the Skorokhod problem with the mean minimum condition. They also studied the stability of the solution. 
In a related study, \cite{briand2020forward} investigated SDE with normal reflection in conditional law and BSDE with normal constraint in law. Using the penalty method, they established the well-posedness of these two methods and proved the forward chaos propagation and backward propagation of chaos \citep{lauriere2022backward}. Additionally, they studied a related obstacle problem in the Wasserstein space. Overall, these studies contribute to a better understanding of MR-BSDE and MR-SDE and their applications in various fields.

The LDP is primarily concerned with asymptotic estimates of the probability of tail events  $A$, especially whether $\p\kuo{X_\varepsilon \in A}$ converges to zero at an exponential rate as $\varepsilon \rightarrow 0$. In this paper, we focus on the Freidlin-Wentzell type of LDP. 
The study of LDP for SDE originates from \cite{Freidlin1984RandomPO}. 
Recently, the weak convergence method \cite{10.1214/07-AOP362} based on the variational representation of Brownian motion \citep{Bou1998AVR} has become an important and effective tool for proving LDP. Using this method, many LDP for SDE and SPDE have been established, see e.g. \citep{jacquier2022large,Rckner2010LargeDF,chiarini2014Large,Zhai2015LargeDF,Ren2010LargeDF,Fatheddin2012LargeDP}. For more on the weak convergence method and LDP, see monograph \citep{budhiraja2019Analysis}. 
\cite{Matoussi2017LargeDP} proposed sufficient conditions for verifying LDP that are easier to verify and used these conditions to prove LDP for quasilinear SPDE with reflection. Currently, there are relatively few studies on MV-SDE. Using the classical exponential equivalence method, \cite{10.1214/18-AAP1416} and \cite{10.1214/07-AAP489} proved LDP for MV-SDE. For reflected MV-SDE, \cite{adams2022large} used the exponential equivalence method to establish the LDP. 
Notably, \cite{liu2022large} proposed a new sufficient condition to verify the criteria of the weak convergence method, and show correct controlled equation for the MV-SDE with jumps, which was used to prove the large deviation principle and moderate deviation principle for the MV-SDE with jumps. 
\cite{li2018large} using weak convergence method shows the LDP of MR-SDE with jump.

The invariant probability measure is an important tool for studying the long-term behavior of SDE. The principle of invariant probability measure for classical SDE or distribution-independent SDE has been extensively studied, as seen in \citep{da1996ergodicity, Qiao2012ExponentialEF, zhang2009exponential, Ren2010ExponentialEO, kulik2009exponential}. 
However, due to the fact that the MV-SDE is not a Markov process, the classical proof method cannot be used. 
\cite{wang2018distribution} first established the $\mathcal{W}_2$-exponential convergence of the solution of the MV-SDE, and subsequently proved the uniqueness and existence of the invariant probability measure. 
\cite{10.1214/20-AIHP1123} studied the exponential ergodicity of MV-SDE with additive pure jump noise in the relative entropy and (weighted) Wasserstein distances. 
\cite{wang2021exponential} studied the exponential ergodicity of singular MV-SDE with or without reflection in the relative entropy and (weighted) Wasserstein distances. \cite{ren2021donsker} studied the Donsker-Varadhan LDP of the invariant measure for path-dependent McKean-Vlasov SPDE.

The Harnack inequality is an important inequality for studying the properties of solutions
 {which implies gradient estimates (hence, the strong Feller property), the uniqueness of invariant probability measures, heat kernel estimates, and the irreducibility of the associated Markov semigroups.}
\cite{Wang1997OnEO} first proposed a dimension-free Harnack inequality for a diffusion process on Riemannian manifolds. If we denote $\left\{P_t\right\}_{t \geq 0}$ is the transition semigroup of a Markov process, the dimension-free Harnack inequality states that  
$$
\left(P_t f\right)^\alpha(x) \leq P_t f^\alpha(x+y) \exp \left\{C_\alpha(y, t)\right\}, x, y \in \mathbb{R}^d,
$$
where $\alpha>1$, $C_\alpha: \mathbb{R}^d \times(0, \infty) \rightarrow \mathbb{R}_{+}$ satisfies $C_\alpha(0, t)=0$, and $x, y \in \mathbb{R}^d$. \cite{WFY_Harnack_AOP} proved the dimension-free Harnack inequality for a large class of stochastic differential equations with multiplicative noise by constructing a coupling with unbounded time-dependent drift. 

This coupling method, also known as the coupling by change measure method, has been used to prove the Harnack inequality for many SDE and stochastic partial differential equation (SPDE for short), see e.g., \citep{10.1214/009117906000001204, niu2019wang, Wang2012HarnackIF, Wang2015HarnackIF, wang2016gradient} and monograph \cite{wang2013harnack}. 
\cite{wang2018distribution} used the coupling by change measure method to prove the Harnack inequality for the MV-SDE when the noise coefficient is independent of the distribution. 
\cite{ren2023singular} studied the Harnack inequality when the drift contains a locally standard integrable term and a Lipschitz continuous term in the spatial variable and is Lipschitz continuous in the distribution variable with respect to a weighted variation distance. 
 \cite{huang2022log} studied the Harnack inequality for MV-SDE, where the drift satisfies a standard integrability condition in time-space and may be discontinuous in the distance induced by any Dini function. 
 \cite{Wang2012IntegrationBP} developed a novel coupling method to effectively prove the shift Harnack inequality. 
 \cite{huang2021well} proved the shift Harnack inequality for the McKean-Vlasov SPDE in a separable Hilbert space with a Dini continuous drift. 
 \cite{huang2019distribution} used Zvonkin's transform to prove the Harnack inequality and shift Harnack inequality for the MV-SDE when the coefficients are Dini continuous in the spatial variable.

\emph{Our contributions.}

To the best of our knowledge, there are few results on the MR-MVSDE. 
The aim of this paper is to investigate this type of MR-MVSDE. 
In Section \ref{sec:3}, we first establish the existence and uniqueness of a strong solution $(X, K)$ for \eqref{eq:1.1}, followed by a study of its properties, including: 
\begin{enumerate}[(i)]
	\item the relationship between \eqref{eq:1.1} and a PDE with Neumann boundary conditions in the Wasserstein space, i.e., the Feynman-Kac formula (Proposition \ref{pro:F-K}); 
	\item $L^p$ estimates for the solution (Proposition \ref{pro:3.2}); 
	\item the absolute continuity of the Stieljes measure $dK$ with respect to the Lebesgue measure, under certain regularity assumptions on $h$, and an explicit expression for the density (Proposition \ref{pro:3.3}).
\end{enumerate}

In Section \ref{sec:4}, we study the propagation of chaos (Theorem \ref{thm:poc}), i.e., we prove the convergence of an interacting particle system (see \eqref{eq:particle}) to \eqref{eq:1.1}.

In Section \ref{sec:5}, we investigate the impact of small perturbations on the system. Specifically, we study the effects of perturbations in the initial condition (Theorem \ref{thm:sric}), the coefficients (Theorem \ref{thm:src}), and the driving process (Theorem \ref{thm:srdp}). We then use Proposition \ref{pro:LDP1} and Proposition \ref{pro:LDP2} to prove a LDP for \eqref{eq:1.1} (Theorem \ref{thm:LDP}).

In Section \ref{sec:6}, we focus on the long-time behavior of solutions and the Harnack inequality. We prove the existence and uniqueness of an invariant measure for \eqref{eq:1.1} (Theorem \ref{thm:eipm}). Using  the method of coupling by change measure, we prove 
log-Harnack inequality (Theorem \ref{thm:lhi}) and shift Harnack inequality (Theorem \ref{thm:shi}) for \eqref{eq:1.1}.

The paper is organized in the following way: Section \ref{sec:2} introduces the Wasserstein metric and presents some Assumptions and preliminary results. Section \ref{sec:3} shows that the MR-MVSDE admits a unique strong solution $(X,K)$. Then we study the density of $K$.  Section \ref{sec:4} shows that system \eqref{eq:1.1} can be seen as the limit of an interacting particles system with oblique reflection of mean field type. Section \ref{sec:5} studies the effect of small noise on the system. Section \ref{sec:6} proves the existence of invariant probability measure and the Harnack inequality.

Throughout the paper, we adopt the following convention: $\mathcal{B}^+_b(\mathbb{R})$ refers to the space of nonnegative and bounded measurable functions on $\mathbb{R}$. The symbol $C$ denotes various positive constants whose specific values may vary in different contexts. 
\section{Preliminary}\label{sec:2}
In this section, we primarily focus on the Wasserstein space and some preliminary results.
\subsection{The Wasserstein space}
Let $\mathcal{P}\left(\mathbb{R}\right)$ be the space of probability measures on $\mathbb{R}$ and for any $p>1$ denote by $\mathcal{P}_p\left(\mathbb{R}\right)$ the subspace of $\mathcal{P}\left(\mathbb{R}\right)$ of probability measures with finite moment of order $p$. Now, we can defined $p$-Wasserstein distance 
For $\mu, \nu \in \mathcal{P}_p\left(\mathbb{R}\right)$,  the $p$-Wasserstein distance $\mathcal{W}_p(\mu, \nu)$ is
$$
\mathcal{W}_p(\mu, \nu)=\inf _{\pi \in \Pi(\mu, \nu)}\left[\int_{\mathbb{R} \times \mathbb{R}}|x-y|^p d \pi(x, y)\right]^{1 / p}\,,
$$
where $\Pi(\mu, \nu)$ denotes the set of probability measures on $\mathbb{R} \times \mathbb{R}$ whose first and second marginals are, respectively, $\mu$ and $\nu$. An equivalent and more probabilistic definition is 
$$
\mathcal{W}_p(\mu,\nu) = \kuo{\inf_{X\sim \mu, Y \sim \nu}\e\zkuo{\abs{X-Y}^p}}^{\frac{1}{p}}\leq \e\zkuo{\abs{X-Y}^p}^{\frac{1}{p}}\,.
$$
In the following result, we mainly use $1$-Wasserstein distance and $2$-Wasserstein distance. More detail about the Wasserstein metric can be seen \cite{villani2021topics} and references therein.
{
Then, we introduce the Lions' derivative (L-derivative) on $\mathcal{P}_2(\mathbb{R})$.}

{For any $\mu \in \mathcal{P}_2(\mathbb{R})$, let }
{$$
\begin{aligned}
T_{\mu, 2} & :=L^2\left(\mathbb{R}, \mathbb{R} ; \mu\right) \\
& :=\left\{\phi: \mathbb{R} \rightarrow \mathbb{R} ; \phi \text { is measurable with } \mu\left(|\phi|^2\right):=\int_{\mathbb{R}}|\phi(x)|^2 \mu(\mathrm{d} x)<\infty\right\},
\end{aligned}
$$}
{and its norm as following
$$
\|\phi\|_{T_{\mu, 2}}^2:=\int_{\mathbb{R}^d}|\phi|^2 \mu(\mathrm{d} x), \text { for } \phi \in T_{\mu, 2}.
$$
}
{Now, we can define the intrinsic derivative and L-differentiable.}
\begin{definition}[Intrinsic derivative and L-differentiable]
{Let $f: \mathcal{P}_2\left(\mathbb{R}\right) \mapsto \mathbb{R}$ be a continuous function, and $I$ be the identity map on $\mathbb{R}^d$.}
\begin{enumerate}[(i)]
	\item {If for any $\mu \in \mathcal{P}_2\left(\mathbb{R}\right)$
$$
T_{\mu, 2} \ni \phi \rightarrow D_\phi^L f(\mu):=\lim _{\varepsilon \rightarrow 0} \frac{f\left(\mu \circ(I+\varepsilon \phi)^{-1}\right)-f(\mu)}{\varepsilon} \in \mathbb{R}
$$
is a well-defined bounded linear functional, then we say that $f$ is intrinsically differentiable at $\mu$ and denote the intrinsic derivative of $f$ at $\mu$ by $D_\phi^L f(\mu)$.}

	\item {If for any $\mu \in \mathcal{P}_2\left(\mathbb{R}\right)$}
$$
\lim _{\|\phi\|_{\mu, 2} \rightarrow 0} \frac{f\left(\mu \circ(I+\varepsilon \phi)^{-1}\right)-f(\mu)-D_\phi^L f(\mu)}{\varepsilon}=0,
$$
we say that $f$ is L-differentiable at $\mu$ and the denote the L-derivative (i.e. Lions derivative) of $f$ at $\mu$ by $D_\mu f(\mu)$.
\end{enumerate}

\end{definition}
{By the Riesz representation theorem, we know that
$$
\left\langle D_\mu  f(\mu), \phi\right\rangle_{T_{\mu, 2}}:=\int_{\mathbb{R}}\left\langle D^L f(\mu)(x), \phi(x)\right\rangle \mu(\mathrm{d} x)=D_\phi^L f(\mu), \quad \phi \in T_{\mu, 2} .
$$
}

In order to obtain the Feynman-Kac formula for \eqref{eq:1.1}, it is necessary to make certain regularity assumptions on function $f$ and apply a special version of It\^{o}'s lemma, as demonstrated in the work of \cite{briand2020forward}.
\begin{definition}	
We say $f$ is partially $\mathcal{C}^2$ if the map $y \rightarrow$ $D_\mu f(\mu)(y)$ is continuously differentiable with a derivative $(\mu, y) \rightarrow \partial_y D_\mu f(\mu)(y)$ jointly continuous on $\mathcal{P}_2\left(\mathbb{R}\right) \times \mathbb{R}$.
\end{definition}
\begin{lemma}[It\^{o}'s lemma]\label{lem:ito}
	Consider following process
	$$
d X_t=b_t d t+\sigma_t d B_t+ d K_t\,,
$$
where $K_t$ is a deterministic, 
continuous and nondecreasing process with $K_0 = 0$, 
$\left(b_t\right)_{t\geq0}$ and $\left(\sigma_t\right)_{t\geq 0}$ are progressively measurable with values in $\mathbb{R}$ and satisfy following integrable condition:
$$
\mathbb{E}\left[\int_0^T\left(\left|b_s\right|^2+\left|\sigma_s\right|^4\right) d s\right]<+\infty\,.
$$
If $f$ is partially $\mathcal{C}^2$, we have 
$$
\begin{gathered}
f\left(\law{X_t}\right)=f\left(\law{X_0}\right)+\int_0^t \mathbb{E}\left[D_\mu f\left(\law{X_t}\right)\left(X_s\right)b_s\right] d s+\int_0^t \mathbb{E}\left[D_\mu f\left(\law{X_t}\right)\left(X_s\right)\right] d K_s \\
+\frac{1}{2} \int_0^t \mathbb{E}\left[\left(\sigma_s^2\partial_y D_\mu f\left(\law{X_t}\right)\left(X_s\right)\right)\right] d s.
\end{gathered}
$$
\end{lemma}
\subsection{Assumptions and preliminary results}
We make following assumptions about the coefficients:
\begin{assumption}[About $b$, $\sigma$ and $\xi$]\label{ass:A_1}
\begin{enumerate}[(i)]
	\item Lipschitz assumption: there exists a constant $K>0$ such that for all $x,y \in \mathbb{R}$, $\mu,\nu \in \mathcal{P}_2(\mathbb{R})$ we have 
	$$
	|b(x,\mu) - b(y,\nu)| + |\sigma(x,\mu)- \sigma(y,\nu)|\leq K\kuo{|x - y| + \mathcal{W}_2(\mu,\nu)}\,.
	$$
	\item The random variable $\xi$ is square integrable, i.e. $\e\zkuo{\xi^2}<\infty$ and satisfies $\e\zkuo{h(\xi)}\geq 0$.
\end{enumerate}
\end{assumption}
\begin{assumption}[More about $\xi$]\label{ass:xi_4}
	There exists a $p>4$ such that $\e |\xi|^p<\infty$.
\end{assumption}
\begin{assumption}[About $h$]\label{ass:A_2}
The function $h : \mathbb{R}\rightarrow \mathbb{R}$ is increasing and bi-Lipschitz. There exist $0<m\leq M$ such that 
	$$
	\forall x \in \mathbb{R}, \forall y \in \mathbb{R}, \quad m|x-y|\leq |h(x)-h(y)|\leq M|x-y|\,.
	$$
\end{assumption}

\begin{assumption}[About the regularity of $h$]\label{ass:A_4}
	The function $h$ is twice continuously differentiable with bounded derivatives.
\end{assumption}
\begin{assumption}[Dissipative condition]\label{ass:A_5}
	For some constants $C_2 >C_1 \geq 0$, 
$$
2(x-y)(b(x,\mu) - b(y,\nu)) + |\sigma(x,\mu) - \sigma(y,\nu)| \leq C_1 \mathcal{W}_2(\mu,\nu)^2 -C_2 |x-y|^2\,.
$$
\end{assumption}

\begin{remark}
{
The Assumptions \ref{ass:A_1}-\ref{ass:A_5} will be used as follows:}

	{The Assumption \ref{ass:A_1} which is the standard assumption for MV-SDE and Assumption \ref{ass:A_2} used in \cite{briand2020particles,briand2018bsdes,falkowski2021mean,briandMEAN} ensure the well-posedness of \eqref{eq:1.1}. 
The  Assumption \ref{ass:A_4} will be use in show the Proposition \ref{pro:3.3}.}

{Assumption \ref{ass:A_5} is a dissipative condition which will be used to show the existence and uniqueness of invariant probability measure when $t\rightarrow \infty$. }

\end{remark}
We now introduce the functions $H$, $\bar{G}_0$, and $G$ which plays an important role in our paper and give some properties from \cite{briand2020particles}.
\begin{equation*}
H: \mathbb{R} \times \mathcal{P}_1(\mathbb{R}) \ni(x, v) \mapsto \int h(x+z) v(d z)\,.
\end{equation*}

Let $\bar{G}_0$ be the inverse function in space of $H$ evaluated at 0 ,
\begin{equation*}
\bar{G}_0: \mathcal{P}_1(\mathbb{R}) \ni v \mapsto \inf \{x \in \mathbb{R}: H(x, v) \geq 0\}\,,
\end{equation*}
and let $G_0$ denote the positive part of $\bar{G}_0$,
\begin{equation*}
\begin{aligned}
	 G_0: \mathcal{P}_1(\mathbb{R}) \ni \nu \mapsto \inf \{x \geq 0: H(x, v) \geq 0\} \,.
\end{aligned}
\end{equation*}

\begin{lemma}\label{lem:2.1}
	Under Assumption \ref{ass:A_2} we have:
	\begin{enumerate}[(i)]
		\item  For all $\nu$ in $\mathcal{P}_1(\mathbb{R})$, the mapping $H(\cdot, \nu): \mathbb{R} \ni x \mapsto H(x, \nu)$ is a bi-Lipschitz function, namely:
$$
\forall x, y \in \mathbb{R}, \quad m|x-y| \leq|H(x, \nu)-H(y, \nu)| \leq M|x-y| .
$$
	\item For all $x$ in $\mathbb{R}$, the mapping $H(x, \cdot): \mathcal{P}_1(\mathbb{R}) \ni \nu \mapsto H(x, \nu)$ satisfies the following Lipschitz estimate:
$$
\forall v, v^{\prime} \in \mathcal{P}_1(\mathbb{R}), \quad\left|H(x, \nu)-H\left(x, \nu^{\prime}\right)\right| \leq\left|\int h(x+\cdot)\left(d \nu-d \nu^{\prime}\right)\right|\,.
$$
	\end{enumerate}
\end{lemma}
\begin{lemma}\label{lem:2.2}
 Under Assumption \ref{ass:A_2}, the mapping $G_0: \mathcal{P}_1(\mathbb{R}) \ni \nu \mapsto G_0(\nu)$ is Lipschitz continuous in the following sense:
$$
\left|G_0(\nu)-G_0\left(\nu^{\prime}\right)\right| \leq \frac{1}{m}\left|\int h\left(\bar{G}_0(\nu)+\cdot\right)\left(d \nu-d \nu^{\prime}\right)\right|\,,
$$
where we recall that $\bar{G}_0(\nu)$ is the inverse of $H(\cdot, \nu)$ at point $0$. In particular,
$$
\left|G_0(\nu)-G_0\left(\nu^{\prime}\right)\right| \leq \frac{M}{m} \mathcal{W}_1\left(\nu, \nu^{\prime}\right)\,.
$$
\end{lemma}

\section{Existence, uniqueness and relative properties of the solution}\label{sec:3}
In this section, we show the existence and uniqueness of solutions to MR-MVSDE, and establish some properties of solution.
\subsection{Existence and Uniqueness of \eqref{eq:1.1}}
We now introduce the definition of the solution to MR-MVSDE \eqref{eq:1.1}.
\begin{definition}
 A couple of processes $(X, K)$ is said to be a flat deterministic solution to \eqref{eq:1.1} if $(X, K)$ satisfy \eqref{eq:1.1} with $K$ being a non-decreasing continuous deterministic function with $K_0=0$.
\end{definition}
Given this definition, we have the following result.
\begin{theorem}\label{thm:1}
	Under Assumptions \ref{ass:A_1} and \ref{ass:A_2}, the MR-MVSDE \eqref{eq:1.1} has a unique deterministic flat solution $(X,K)$. Moreover, 
\begin{equation*}
\forall t \geq 0, \quad K_t=\sup _{s \leq t} \inf \left\{x \geq 0: \mathbb{E}\left[h\left(x+U_s\right)\right] \geq 0\right\}=\sup _{s \leq t} G_0\left(\law{U_s}\right),
\end{equation*}
where $\kuo{U_t}_{0\leq t\leq T}$ is the process defined by :
\begin{equation}\label{eq:U}
\begin{aligned}
	  U_t = \xi + \int_0^t b(X_s,\mathcal{L}_{X_s})ds + \int_0^t \sigma(X_s,\mathcal{L}_{X_s})dB_s.
\end{aligned}
\end{equation}
With these notations, denoting by $\law{U_s}$ the family of marginal laws of $\kuo{U_t}_{0\leq t\leq T}$ we have 
$$
K_t = \sup_{s\leq t}G_0(\law{U_s}).
$$
\end{theorem}
\begin{remark}
	From this construction, we deduce that for all positive $r$:
 \begin{align}
 \notag
    & K_{t+r}-K_t \\
     \notag
 &= \sup_{s\in [0,r]}\inf\dkuo{x\geq 0:\e\zkuo{h\kuo{x + X_t + \int_t^{t+s}b\kuo{X_u,\law{{X}_u}}du+\int_t^{t+s}\sigma\kuo{X_u,\law{{X}_u}}dB_u}}\geq 0
}	\,.
 \end{align}

\end{remark}

The following Proposition establishes the connection of  MR-MVSDE \eqref{eq:1.1} and a PDE with Neumann boundary condition on the Wasserstein space. 

Let $\mathcal{O}: = \dkuo{\mu \in \mathcal{P}_2(\mathbb{R}),\int h(x)\mu(dx)>0}$. Given a bounded, continuous map $G:\bar{\mathcal{O}} \rightarrow \mathbb{R}$ we consider the map $u: [0,T]\times \bar{\mathcal{O}} \rightarrow \mathbb{R}$ defined by 
\begin{equation}\label{eq:FK_u}
\begin{aligned}
	  u(t,\mu) = \e\zkuo{G(\mathcal{L}_{X_{t}})}\,.
\end{aligned}
\end{equation}

\begin{proposition}[Feynman-Kac formula]\label{pro:F-K}
Suppose that Assumptions \ref{ass:A_1} and \ref{ass:A_2} hold. 
	Assume that $u$ is a classical solution to Neumann problem on the Wasserstein space \eqref{eq:FK_f}. Then, $u$ is given by \eqref{eq:FK_u}.

\begin{equation}\label{eq:FK_f}
\left\{\begin{array}{l}
\text { (i) } \left(\partial_t+\mathcal{T}\right) u(t, \mu)=0 \quad \text { in }(0, T) \times \mathcal{O}, \\
\text { (ii) } \int_{\mathbb{R}} {D}_\mu u(t, \mu)(y)  \mu(d y)=0, \quad \text { in }(0, T)\times \partial \mathcal{O}, \\
\text { (iii) } u(T, \mu)=G(\mu), \quad \text { in } \mathcal{O},
\end{array}\right.
\end{equation}
where the operator $\mathcal{T}$ is given, for any smooth function $\phi:[0,T]\times \mathcal{P}_2(\mathbb{R})\rightarrow \mathbb{R}$, by
\begin{equation*}
\begin{aligned}
	  \mathcal{T}\phi(t,\mu) := \int_{\mathbb{R}}b(z,\mu) D_\mu\phi(t,\mu)(z)d\mu(z) + \frac{1}{2}\int_{\mathbb{R}}\sigma(z,\mu)^2\partial_z D_\mu \phi(t,\mu)(z) d\mu(z)\,.
\end{aligned}
\end{equation*}
\end{proposition}
\begin{proposition}\label{pro:3.2}
Suppose that Assumptions \ref{ass:A_1} and \ref{ass:A_2} hold.  
\begin{enumerate}[(i)]
	\item  For all $p$, there exists a positive $C_p$, depending on $T, b, \sigma$ and $h$ such that 
$$
\e\zkuo{\sup_{t\leq T}|X_t|^p} \leq C_p(1+\e\zkuo{|\xi|^p})\,.
$$

	\item There exists a positive constant $C$, depending on $T, b, \sigma$ and $h$ such that 
$$
\forall 0\leq s\leq t\leq T ,\quad |K_t-K_s| \leq C|t-s|^{\frac{1}{2}}\,.
$$
	\item For all $p\geq 1$ such that $\xi \in L^p$ we have 
$$
\forall 0\leq s\leq t\leq T ,\quad \e\zkuo{\abs{X_t-X_s}^p}\leq C|t-s|^{\frac{p}{2}}\,.
$$
\end{enumerate}	
\end{proposition}
\subsection{Density of $K_t$}
Let $\mathcal{A}_\mu$ be the partial operator of second order defined by 
\begin{equation*}
\begin{aligned}
	  \mathcal{A}_\mu f(x) = b(x,\mu)\frac{\partial }{\partial x} f(x) + \frac{1}{2}\sigma\kuo{x,\mu}^2\frac{\partial^2}{\partial x^2}f(x)\,.
\end{aligned}
\end{equation*}

\begin{proposition}\label{pro:3.3}
Suppose Assumptions \ref{ass:A_1}, \ref{ass:A_2} and \ref{ass:A_4} hold. Let $(X,K)$ be the unique deterministic flat solution to \eqref{eq:1.1}. Then the process $K$ is Lipschitz continuous and the Stielije measure $dK$ has the following density w.r.t the Lebesgue measure: 
	\begin{equation}\label{eq:density_dk}
k: \mathbb{R}^{+} \ni t \longmapsto \frac{\left(\mathbb{E}\left[\mathcal{A}_{\mathcal{L}_{X_t}} h\left(X_{t}\right)\right]\right)^{-}}{\mathbb{E}\left[h^{\prime}\left(X_{t}\right)\right]} \mathbbm{1}_{\mathbb{E}\left[h\left(X_t\right)\right]=0}\,,
\end{equation}
where, $(\cdot)^{-}:= min(\cdot,0)$.
\end{proposition}

\subsection{Proofs of Theorem \ref{thm:1}, Proposition \ref{pro:F-K} and Proposition \ref{pro:3.3}}
\begin{proof}[Proof of Theorem \ref{thm:1}]
	The proof for the case of MR-SDE is given in \cite{briand2020particles}. We present here the proof of the MR-MVSDE.
	
\textbf{Step 1: Construct the solution $(X,K)$ of \eqref{eq:1.1}.}

	Let us consider the set 
 $$\mathcal{C}^2 = \dkuo{X_t \text{ is } \mathcal{F}_t-\text{adapted continuous process}\,, \e\kuo{\sup_{0\leq t\leq T}|X_t|^2< \infty}} \,.$$
 And let $\widetilde{X}_t\in \mathcal{C}^2$ be a given process. We set 
\begin{equation}
\tilde{U}_t=\xi+\int_0^t b\left(\tilde{X}_s, \mathcal{L}_{\tilde{X}_s}\right) d s+\int_0^t \sigma\left(\tilde{X}_s, \mathcal{L}_{\tilde{X}_s}\right) d B_s, \quad \operatorname{Law}\left(\tilde{U}_s\right)=: \law{\tilde {U}_s}\,,
\end{equation}
and define the function $K_t$ by setting
$$
{K_t=\sup _{s \leq t} \inf \left\{x \geq 0: \mathbb{E}\left[h\left(x+\tilde{U}_s\right)\right] \geq 0\right\}=G_0\left(\law{\tilde {U}_s}\right)}\,.
$$
The function $K_t$ being given, let us define the process $X_t$ by the formula
\begin{equation}
\begin{aligned}
	  X_t=\xi+\int_0^t b\left(\tilde{X}_s,\mathcal{L}_{\tilde{X}_s}\right) d s+\int_0^t \sigma\left(\tilde{X}_s,\mathcal{L}_{\tilde{X}_s}\right) d B_s+K_t .
\end{aligned}
\end{equation}

\textbf{Step 2: Check $K_t$ satisfies $\mathbb{E}\left[h\left(X_t\right)\right] \geq 0$ and $\int_0^t \mathbb{E}\left[h\left(X_s\right)\right] d K_s=0$.}

Based on the definition of $K_t$, we have $\e\zkuo{h(X_t)}\geq 0$, $K_t = G_0\kuo{\law{\tilde{U}_t}}, dK_t-a.e.$ and $G_0(\law{\tilde{U}_t})>0, dK_t-a.e.$ Then we obtain 
$$
\begin{aligned}
\int_0^t \mathbb{E}\left[h\left(X_s\right)\right] d K_s & =\int_0^t \mathbb{E}\left[h\left(\tilde{U}_s+K_s\right)\right] d K_s \\
& =\int_0^t \mathbb{E}\left[h\left(\tilde{U}_s+G_0\left(\law{\tilde{U}_s}\right)\right)\right] d K_s \\
& =\int_0^t \mathbb{E}\left[h\left(\tilde{U}_s+G_0\left(\law{\tilde{U}_s}\right)\right)\right] \mathbbm{1}_{\dkuo{G_0\left(\law{\tilde{U}_s}\right)>0}} d K_s .
\end{aligned}
$$
Moreover, since $h$ is continuous, we have $\mathbb{E}\left[h\left(\tilde{U}_s+G_0\left(\law{\tilde{U}_s}\right)\right)\right]=0$ as soon as $G_0\left(\law{\tilde{U}_s}\right)>0$, so that
$$
\int_0^t \mathbb{E}\left[h\left(X_s\right)\right] d K_s=0\,.
$$

\textbf{Step 3: Show the solution map $\Phi$ which associates to $\tilde{X}$ to the solution $X$ of \eqref{eq:1.1} is a contraction.}

 Let $\tilde{X}, \tilde{X}^\prime \in \mathcal{C}^2$ and define $\tilde{U}$, $K$ and $\tilde{U}^\prime$, $K^\prime$ as above, using the some Brownian motion. We have from the Lipschitz assumption of $b$ and $\sigma$, Cauchy-Schwarz inequality and Burkholder-Davis-Gundy (B-D-G) inequality 
 \begin{equation}\label{eq:Th_1.1}
\begin{aligned}
	  &\e\zkuo{\sup_{t\leq T}|\tilde{X}_t - \tilde{X}_t^\prime|^2}\\
	  &\leq  3\e \left[\sup_{t\leq T}\left\{\abs{\int_0^t\kuo{b\kuo{\tilde{X}_s,\mathcal{L}_{\tilde{X}_s}}- b\kuo{\tilde{X}^\prime_s,\mathcal{L}_{\tilde{X}^\prime_s}} 
	  }ds}^2 + \abs{\int_0^t \kuo{\sigma\kuo{\tilde{X}_s,\mathcal{L}_{\tilde{X}_s}}-\sigma\kuo{\tilde{X}^\prime_s,\mathcal{L}_{\tilde{X}^\prime_s}}}dB_s}^2\right.\right.\\
	  &\left.\left.\qquad +\abs{K_t - K^\prime_t}^2 \right\}\right]\\
	  &\leq3 \mathbb{E}\left[\left.T\cdot \sup _{t \leq T} \int_0^t\left(b\left(\tilde{X}_s, \mathcal{L}_{\tilde{X}_s}\right)-b\left(\tilde{X}_s^{\prime}, \mathcal{L}_{\tilde{X}_s^{\prime}}\right)\right) d s\right|^2\right]
   \\
	  &\qquad
   +3 \mathbb{E}\left[\sup _{t \leq T}\left|\int_0^t\left(\sigma\left(\tilde{X}_s, \mathcal{L}_{\tilde{X}_s}\right)-\sigma\left(\tilde{X}^{\prime}_s, \mathcal{L}_{\tilde{X}^{\prime}_s}\right)\right) d B_s\right|^2\right] + 3 \sup_{t\leq T}\abs{K_t - K^\prime_t}^2 \\
	  & \leq CT\dkuo{T\e\zkuo{\sup_{t\leq T}\abs{\tilde{X}_t - \tilde{X}^\prime_t}}+\e\zkuo{\int_0^T \mathcal{W}_2\kuo{\mathcal{L}_{\tilde{X}_s},\mathcal{L}_{\tilde{X}^\prime_s}}^2ds }}\\
	  &\qquad+ CT \e\zkuo{\sup_{t\leq T}\abs{\tilde{X}_t - \tilde{X}^\prime_t}}+CT\e\zkuo{\int_0^T \mathcal{W}_2\kuo{\mathcal{L}_{\tilde{X}_s},\mathcal{L}_{\tilde{X}^\prime_s}}^2ds }+ 3 \sup_{t\leq T}\abs{K_t - K^\prime_t}^2   \,.
\end{aligned}
\end{equation}
From the representation of $K$ and Lemma \ref{lem:2.2}, we have 
\begin{equation}\label{eq:K_t}
\begin{aligned}
	  \sup_{t\leq T}|K_t-K^\prime_t|^2 &= \sup_{t\leq T}\abs{\sup_{s\leq t}G_0(\law{\tilde X_t})-\sup_{s\leq t}G_0(\law{\tilde X_t^\prime})}\leq \sup_{t\leq T}\abs{G_0(\law{\tilde X_t})-G_0(\law{\tilde X_t^\prime})}\\
	  &\leq \frac{M}{m} \mathbb{E}\left[\sup _{t \leq T}\left|\tilde{U}_t-\tilde{U}^{\prime}_t\right|^2\right] \leq  C\mathbb{E}\left[\sup _{t \leq T}\left|\tilde{X}_t-\tilde{X}_t^{\prime}\right|^2\right] \,.
\end{aligned}
\end{equation}
Plugging \eqref{eq:K_t} into \eqref{eq:Th_1.1}, we have 
$$
\e\zkuo{\sup_{t\leq T}|\tilde X_t - \tilde X_t^\prime|}\leq  C(1+T) T \mathbb{E}\left[\sup _{t \leq T}\left|\tilde{X}_t-\tilde{X}_t^{\prime}\right|^2\right].
$$
Hence, there exists a positive $\tilde{T}$, depending on $b, \sigma$ and $h$ only, such that for all $T<\tilde{T}$, the map $\Phi$ is a contraction. We first deduce the existence and uniqueness of the solution on $[0, \tilde{T}]$ and then on $\mathbb{R}^{+}$ by iterating the construction. 

\end{proof}

\begin{proof}[Proof of Proposition \ref{pro:F-K}]
	For given $X_0\in \mathbb{R}$ and $t_0\geq 0$, let us set $\mu_s : = \mathcal{L}_{X_s^{t_0,X_0}}$. Then, by Lemma \ref{lem:ito}, 
	$$
	u(T,\mu_T) = u(t_0,\mu_{t_0}) + \int_{t_0}^T\kuo{\partial_t + \mathcal{T}}u(s,\mu_s) + \int_{t_0}^T\int_\mathbb{R}D_\mu u(s,\mu_s)(y) \mu_s(dy) dK_s 	\,,
 $$
	where, by \eqref{eq:FK_f}-(ii), 
	$$
	\int_{t_0}^T\int_{\mathbb{R}}D_\mu u(s,\mu_s)(y) \mu_s(dy)dK_s = \int_{t_0}^T\mathbbm{1}_{\dkuo{\mu_s \in \partial \mathcal{O}}} \kuo{\int_{\mathbb{R}}D_\mu u(s,\mu_s) (y) }dK_s = 0\,.
	$$
	From the equation satisfied by $u$, we obtain
	$$
	u(t_0,\mu_0) = u(t_0,\mu_{t_0}) = \e\zkuo{u(T,\mu_T)} = \e\zkuo{G(\mathcal{L}_{X_T})}\,.
	$$
\end{proof}

\begin{proof}[Proof of Proposition \ref{pro:3.2}]

	(i) We have 
	\begin{equation}\label{eq:pro_3_2_1}
\begin{aligned}
\mathbb{E}\left[\sup _{t \leq T}\left|X_t\right|^p\right] & \leq 4^{p-1}\left\{\mathbb{E}\left[|\xi|^p\right]+\mathbb{E}\left[\sup _{t \leq T}\left(\int_0^t\left|b\left(X_s, \mathcal{L}_{X_s}\right)\right| d s\right)^p\right]\right. \\
& \left.+\mathbb{E}\left[\sup _{t \leq T}\left(\int_0^t\left|\sigma\left(X_s, \mathcal{L}_{X_s}\right)\right| d B_s\right)^p\right]+K_T^p\right\}\,.
\end{aligned}
\end{equation}
We first deal with $K^p_T = \sup_{t\leq T}G_0(\law{U_t})$. Apply Lipschitz property of $G_0$, Lemma \ref{lem:2.2} and the definition of the Wasserstein metric $\mathcal{W}_1$, we have 
$$
\forall t \geq 0, \quad \abs{G_0(\law{U_t})}\leq \frac{M}{m}\e\zkuo{|U_t - U_0|},
$$
since $G_0(\law{U_0}) = 0$ as $\e\zkuo{h(X_0)}\geq 0$, and where $U$ is defined by \eqref{eq:U}. Therefore 
\begin{equation}\label{eq:pro_3_2_2}
\begin{aligned}
	  |K_T|^p = \abs{\sup_{t\leq T}G_0(\law{U_t})}^p & \leq  2^{p-1}\left\{ \kuo{\frac{M}{m}}^p\e\zkuo{\sup_{t\leq T}\kuo{\int_0^t\abs{b(X_s,\mathcal{L}_{X_s})}ds}^p}\right . \\
	  &\qquad \left .+ \e\zkuo{\sup_{t\leq T}\kuo{\int_0^t\abs{\sigma(X_s,\mathcal{L}_{X_s})}dB_s}^p}\right\}\,.
\end{aligned}
\end{equation}
Plugging \eqref{eq:pro_3_2_2} into \eqref{eq:pro_3_2_1}, we deduce that there exists a constant $C_p\geq 0$ such that 
\begin{equation}
\begin{aligned}
	  \e\zkuo{\sup_{t\leq T}|X_t|^p}&\leq C_p \e\zkuo{|\xi|^p + \sup_{t\leq T}\kuo{\int_0^t\abs{b(X_s,\mathcal{L}_{X_s})}ds}^p + \sup_{t\leq T}\abs{\int_0^t \sigma\kuo{X_s,\mathcal{L}_{X_s}}dB_s}^p}\\
	  &\leq C_p\dkuo{\e\zkuo{|\xi|^p} + T^{p-1}\e\zkuo{\int_0^T(1+|X_s|)^pds} + C\e\zkuo{\int_0^T(1+|X_s|)^2ds}^{\frac{p}{2}}}\\
	  &\leq C_p(1+\e\zkuo{|\xi|^p}) + C \int_0^T\e\zkuo{\sup_{t\leq r}|X_t|^p}dr\,.
\end{aligned}
\end{equation}
Apply Gronwall's lemma, we deduce for all $p\geq 2$, there exists a constant $C>0$ such that 
$$
\e\zkuo{\sup_{t\leq T}|X_t|^p} \leq C(1+\e\zkuo{|\xi|^p}).
$$

(ii) we observe that $\nu \rightarrow G_0(\nu)$ is Lipschitz continuous, we have 
$$
K_t - K_s = \sup_{r\leq t}G_0(\law{U_r}) -\sup_{r\leq s}G_0(\law{U_r})\leq \frac{M}{m}\mathcal{W}_1(\law{U_r},\law{U_s})\leq \frac{M}{m}\e\zkuo{\abs{U_r - U_s}},
$$
so the result follows from standard computations.

(iii) we use the decomposition 
$$X_t - X_s = U_t - U_s + K_t - K_s\,,
$$
As a result of (ii), we complete this proof.
\end{proof}

In order to prove Proposition \ref{pro:3.3}, we need following lemma: 
\begin{lemma}\label{lem:3.3}
If Assumptions \ref{ass:A_1}, \ref{ass:A_2} and \ref{ass:A_4} hold.
	\begin{enumerate}[(i)]
		\item If $\varphi$ is a continuous function such that, for some $C\geq 0$ and $p\geq 1$, 
		$$
		\forall x \in \mathbb{R}, \quad |\varphi(x)|\leq C(1+|x|^p)\,,
		$$
then the function $t\rightarrow \e\zkuo{\varphi(X_t)} $ is continuous.
		\item The functions $t \rightarrow  \e\zkuo{h(X_t)}$ and $t \rightarrow \e\zkuo{\mathcal{A}_{\mathcal{L}_{X_r}}h(X_t)}$ are continuous.
	\end{enumerate}
\end{lemma}
\begin{proof}
	(i) is the \cite{briandMEAN}[Lemma 3]. 
	
	Now, we prove (ii). Since the Assumptions \ref{ass:A_1}, \ref{ass:A_2} and \ref{ass:A_4} hold., the mapping $x\longmapsto \mathcal{A}_{\mathcal{L}_{X_t}}h(X_t)$ is continuous and for all $x\in \mathbb{R}$, there exists some universe constants $C_1$ and $C_2$ such that 
	$$
	\begin{aligned}
		\abs{b(x,\mu)h^\prime(x)}\leq C_1(1+|x|)\,,\\
		\abs{\sigma(x,\mu)^2h^{\prime}(x)}\leq C_2(1+|x|).
	\end{aligned}
	$$
	Then, using (i), we deduce that $t \rightarrow \e\mathcal{A}_{\mathcal{L}_{X_t}}h(X_t)$ is continuous.
\end{proof}
\begin{proof}[Proof of Proposition \ref{pro:3.3}]
		Firstly, we show $K$ is Lipschitz continuous. To do so, we prove that $s\mapsto \bar{G}_0(\mu_s) $ is Lipschitz continuous on $[0,T]$. If $s \leq t$, we have 
$$
\begin{aligned}
\left|\bar{G}_0\left(\law{U_s}\right)-\bar{G}_0\left(\law{U_t}\right)\right| & \leq \frac{1}{m}\left|H\left(\bar{G}_0\left(\law{U_s}\right), \law{U_t}\right)\right| \\
& =\frac{1}{m}\left|\mathbb{E}\left[h\left(\bar{G}_0\left(\law{U_s}\right)+U_t\right)\right]\right| \\
& =\frac{1}{m}\left|\mathbb{E}\left[h\left(\bar{G}_0\left(\law{U_s}\right)+U_s+\int_s^t b\left(X_r,\mathcal{L}_{X_r}\right) d r+\int_s^t \sigma\left(X_r,\mathcal{L}_{X_r}\right) d B_r\right)\right]\right| ,
\end{aligned}
$$
the first equality since $H(\bar{G}_0(\mu_t),\mu_t) = 0$.

Let $\bar{A}_{x,\mu}$ be the partial operator of second order defined by 
$$
\bar{\mathcal{A}}_{y,\mu}f(x) := b(y,\mu) \frac{\partial }{\partial x}f(x) + \frac{1}{2}\sigma\kuo{y,\mu}^2\frac{\partial^2}{\partial x^2}f(x)\,.
$$
Apply It\^{o}'s lemma, we have 
\begin{equation}
\begin{aligned}
	  \e\zkuo{h\kuo{\bar{G}_0(\law{U_s}) + U_t}} &= \e\zkuo{h\kuo{\bar{G}_0(\law{U_s})+U_s}} + \int_s^t\e\zkuo{\bar{\mathcal{A}}_{X_r,\mathcal{L}_{X_r}}h\kuo{\bar{G}_0(\law{U_s}) + U_r}}dr \\
	  &=H\left(\bar{G}_0\left(\law{U_s}\right), \law{U_s}\right)+\int_s^t\e\zkuo{\bar{\mathcal{A}}_{X_r,\mathcal{L}_{X_r}}h\kuo{\bar{G}_0(\law{U_s}) + U_r}}dr\\
	  & = \int_s^t\e\zkuo{\bar{\mathcal{A}}_{X_r,\mathcal{L}_{X_r}}h\kuo{\bar{G}_0(\law{U_s}) + U_r}}dr.
\end{aligned}
\end{equation}
The result follows that $h$ has bounded derivatives and $\sup_{t\leq T}|X_t|\in L^2$ for each $T>0$. So from \cite{briand2020particles}, we deduce that $K$ is Lipschitz continuous and has a bounded density on $[0,T]$ for each $T>0$. 

Then we show the density of $dK$ is \eqref{eq:density_dk}. From Assumption \ref{ass:A_4} and apply It\^{o}'s lemma, for any $s\leq t$ we can get
\begin{equation*}
\begin{aligned}
	  h(X_t) - h(X_s) &= \int_s^tb(X_r,\mathcal{L}_{X_r})h^{\prime}(X_r)dr + \int_s^t \sigma (X_r,\mathcal{L}_{X_r})h^{\prime}(X_r)dB_r \\
	  &\qquad+ \frac{1}{2}\int_s^t \sigma (X_r,\mathcal{L}_{X_r})^2h^{\prime}(X_r)d r  + \int_s^t h^{\prime}(X_r)dK_r \\
	  & =\int_s^t \mathcal{A}_{\mathcal{L}_{X_r}}h(X_r)dr + \int_s^t h^\prime(X_r)dK_r + \int_s^t \sigma\kuo{X_r,\mathcal{L}_{X_r}}dB_r \,.
\end{aligned}
\end{equation*}
Then, we obtain 
\begin{equation}\label{eq:pro_3.4_1}
\begin{aligned}
	  \e\zkuo{\int_s^t h^\prime(X_r)dK_r} = \e h(X_t) - \e h(X_s) - \int_s^t \e \mathcal{A}_{\mathcal{L}_{X_r}}h(X_r)dr\,.
\end{aligned}
\end{equation}

Using \eqref{eq:pro_3.4_1}, Lemma \ref{lem:3.3} , and the proof of \cite{briand2020particles}[Proposition 2.7], we can get the density of $dK$ w.r.t the Lebesgue measure is 
$$
k_t = \frac{\left(\mathbb{E}\left[\mathcal{A}_{\mathcal{L}_{X_t}} h\left(X_{t}\right)\right]\right)^{-}}{\mathbb{E}\left[h^{\prime}\left(X_{t}\right)\right]} \mathbbm{1}_{\mathbb{E}\left[h\left(X_t\right)\right]=0}\,.
$$
Then, we complete the proof.
\end{proof}
\section{Propagation of Chaos}\label{sec:4}
We can write the unique solution of the MR-MVSDE \eqref{eq:1.1} as :
\begin{equation}\label{eq:cou_1}
\begin{aligned}
X_t=\xi+\int_0^t b\left(X_s, \mathcal{L}_{X_s}\right) d s+\int_0^t \sigma\left(X_s,\mathcal{L}_{X_s}\right) d B_s+ \sup_{s\leq t}G_0(\law{U_s})\,,
\end{aligned}
\end{equation}
where $\mathcal{L}_{U_s}$ stands for the law of 
$$
U_t = \xi+\int_0^t b\left(X_s, \mathcal{L}_{X_s}\right) d s+\int_0^t \sigma\left(X_s,\mathcal{L}_{X_s}\right) d B_s\,.
$$
We are here interested in the particle approximation of such a system. Our candidates are the particles
\begin{equation}\label{eq:particle}
\begin{aligned}
	  X^{i}_t = \hat{\xi} + \int_0^t b\left(X_s, \frac{1}{n}\sum_{j = 1}^n \delta_{x^j_s}\right) d s+\int_0^t \sigma\left(X_s,\frac{1}{n}\sum_{j = 1}^n \delta_{x^j_s}\right) d B^i_s + \sup_{s\leq t}G_0\kuo{\law{U^{\xi,n}_s}}\,,
\end{aligned}
\end{equation}
where $B^i$ are independent Brownian motions, $\hat{\xi}^i = \xi^i + G_0(\law{U^{\xi}})$, $(\xi^i)_i $ are independent copies of $\xi$, $\law{U^{\xi}_s}$ denotes the empirical distribution at time $s$ of the particles
$$
U_s^{\xi,i}=\hat{\xi}^i+\int_0^s b\left(X_r^i,\frac{1}{n}\sum_{j = 1}^n \delta_{x^j_s}\right) d r+\int_0^s \sigma\left(X_r^i,\frac{1}{n}\sum_{j = 1}^n \delta_{x^j_s}\right) d B_r^i\,, 1 \leq i \leq N\,, \law{U^{\xi}_s}=\frac{1}{n} \sum_{i=1}^n \delta_{U_s^{\xi,i}}\,.
$$
\begin{remark}

	Similar \cite{briand2020particles} and \cite{briandMEAN}, the previous system of interacting particles can be seen as a multidimensional reflected SDE with oblique reflection. If $h$ is concave, the set
$$
\mathcal{S} := \dkuo{\kuo{x_1,\cdots,x_n}\in \mathbb{R}^n: h(x_1) + \cdots + h(x_n) \geq 0}
$$
 is convex and the system 
 \begin{equation*}
\left\{\begin{array}{l}
  X^{i}_t = \hat{\xi} + \int_0^t b\left(X^i_s, \frac{1}{n}\sum_{j = 1}^n \delta_{x^j_s}\right) d s+\int_0^t \sigma\left(X^i_s,\frac{1}{n}\sum_{j = 1}^n \delta_{x^j_s}\right) d B^i_s + K^n_t, \quad 1\leq i\leq n\\
\frac{1}{n}\sum_{i = 1}^n h(X^i_t)\geq 0,\quad \frac{1}{n}\sum_{i = 1}^n\int_0^th(X^i_t)dK^n_s = 0
\end{array}\right.
\end{equation*}
is the SDE with oblique reflection in convex $\mathcal{S}$.
	\end{remark}
 We use the synchronous coupling technique to prove the propagation of chaos. Specifically, we consider the following system of synchronous coupling:
	$$
		{\begin{cases}
		\bar{X}^i_t = \xi^i + \int_0^tb(\bar{X}^i_s,\law{\bar X^i_s})ds+\int_0^t \sigma(\bar{X}^i_s,\law{\bar {X}^i_s})dB^i_s +\sup_{s\leq t}G_0(\law{\bar {U}^i_s})\\
		X^i_t = \hat\xi^i + \int_0^t b(X^i_s,\frac{1}{n}\sum_{j=1}^n\delta_{X^j_s})ds + \int_0^t \sigma(X^i_s, \frac{1}{n}\sum_{j=1}^n\delta_{X^j_s})dB^i_s+\sup_{s\leq t}G_0(\law{U^{\xi,n}_s})\\
		\bar{U}^{i}_t = \xi^i + \int_0^t b(\bar{X}^i_s,\law{\bar{X}^i_s})ds+ \int_0^t \sigma(\bar{X}^i_s,\law{\bar {X}^i_s})dB^i_s \\
		U^{\xi,i}_t = \hat{\xi}^i + \int_0^t b\left(X^i_s, \frac{1}{n}\sum_{j = 1}^n \delta_{x^j_s}\right) d s+\int_0^t \sigma\left(X^i_s,\frac{1}{n}\sum_{j = 1}^n \delta_{x^j_s}\right) d B^i_s\\
		\law{\bar{U}^i_t} = \operatorname{Law}(\bar{U}^{i}_t),\quad \law{U^{\xi}_t} = \frac{1}{n}\sum_{j=1}^n \delta_{U^{\xi,i}_t}\,,
	\end{cases}}
		$$
		where $\xi^i$ is independent copy of $\xi$.
\begin{theorem}[Propagation of Chaos]\label{thm:poc}
	Let $T>0$, and suppose that Assumptions \ref{ass:A_1}, \ref{ass:xi_4} and \ref{ass:A_2}   hold. 
	there exists a constant $C$ depending on $b$, $\sigma$, $M$, $m$ and $T$ such that, for each $j \in \dkuo{1,\cdots,n}$, 
		$$
		\e\zkuo{\sup_{t\leq T}\abs{X^j_t - \bar{X}^j_t}^2}\leq C n^{-\frac{1}{2}}\,.
		$$

\end{theorem}
\begin{remark}
	Unlike in \cite{briand2020particles} [Theorem 3.3], improving the regularity of $h$ does not improve the convergence speed of propagation of chaos. This is because in this case, the drift coefficient and diffusion coefficient are both related to the law of the solution. We need to control $$
|b(\bar{X}^i_t,\law{\bar{X}^i_s}) - b({X}^i_t,\frac{1}{n}\sum_{j=1}^n \delta_{X^j_s})|^2+|\sigma(\bar{X}^i_t,\law{\bar{X}^i_s}) - \sigma({X}^i_t,\frac{1}{n}\sum_{j=1}^n \delta_{X^j_s})|^2
$$ and the convergence speed of this part is $n^{-\frac{1}{2}}$.
\end{remark}
\begin{proof}
	Let $t >0$, for $r \leq t$, we have 
\begin{equation}
\begin{aligned}
\left|X_r^j-\bar{X}_r^j\right| \leq & \left|\hat{\xi}^j-\xi^j\right| \\
& +\int_0^r\left|b\left(X_s^j, \frac{1}{n}\sum_{j = 1}^n \delta_{x^j_s}\right)-b\left(\bar{X}_s^j,\mathcal{L}_{\bar{X}^j_s}\right)\right| d s\\
&+\left|\int_0^r\left(\sigma\left(X_s^j,\frac{1}{n}\sum_{j = 1}^n \delta_{x^j_s}\right)-\sigma\left(\bar{X}_s^j,\mathcal{L}_{\bar{X}_s}\right)\right) d B_s^j\right| \\
& +\left|\sup _{s \leq r} G_0\left(\law{U^{\xi}_s}\right)-\sup _{s \leq r} G_0\left(\law{\bar{U}^i_s}\right)\right| .
\end{aligned}
\end{equation}

From the inequality 
$$
\begin{aligned}
\left|\sup _{s \leq r} G_0\left(\law{U^{\xi}_s}\right)-\sup _{s \leq r} G_0\left(\law{\bar{U}^i_s}\right)\right| & \leq \sup _{s \leq r}\left|G_0\left(\law{U^{\xi}_s}\right)-G_0\left(\law{\bar{U}^i_s}\right)\right| \leq \sup _{s \leq t}\left|G_0\left(\law{U^{\xi}_s}\right)-G_0\left(\law{\bar{U}^i_s}\right)\right| \\
& \leq \sup _{s \leq t}\left|G_0\left(\law{U^{\xi}_s}\right)-G_0\left(\law{\tilde{U}_s}\right)\right|+\sup _{s \leq t}\left|G_0\left(\law{\tilde{U}_s} \right)-G_0\left(\law{\bar{U}^i_s}\right)\right|,
\end{aligned}
$$
where $\law{\tilde{U}_s}:=\frac{1}{n}\sum_{i=1}^n\delta_{\bar{U}^i_s}$.
We have 
\begin{equation*}
\begin{aligned}
	  \sup_{r\leq t}\abs{X^j_r - \bar{X}^j_r}\leq I^j_1(r) +  \sup _{s \leq t}\left|G_0\left(\law{U^N_s}\right)-G_0\left(\law{\tilde{U}_s}\right)\right|+\sup _{s \leq t}\left|G_0\left(\law{\tilde{U}_s}\right)-G_0\left(\law{\bar{U}^j_s}\right)\right|,
\end{aligned}
\end{equation*}
where, $I^j_1$ is defined by 
\begin{equation*}
\begin{aligned}
	  I^j_1 = \int_0^t \abs{b\kuo{X^j_s, \frac{1}{n}\sum_{j=1}^n \delta_{X^j_s}}-b\kuo{\bar{X}^j_s,\mathcal{L}_{\bar{X}^j_s}}}ds + \abs{\int_0^t \sigma\kuo{X^j_s, \frac{1}{n}\sum_{i=1}^n \delta_{X^j_s}}-\sigma\kuo{\bar{X}^j_s,\mathcal{L}_{\bar{X}^j_s}}}dB^j_s\,.
\end{aligned}
\end{equation*}
We have 
\begin{equation*}
\begin{aligned}
	  \e\zkuo{|I^j_1(t)|^2} &\leq C\e\zkuo{t\int_0^t\abs{b\kuo{X^j_s, \frac{1}{n}\sum_{j=1}^n \delta_{X^j_s}}-b\kuo{\bar{X}^j_s,\mathcal{L}_{\bar{X}^j_s}}}^2ds}\\
	  &\qquad +C\e\zkuo{\int_0^t\abs{ \sigma\kuo{X^j_s, \frac{1}{n}\sum_{j=1}^n \delta_{X^j_s}}-\sigma\kuo{\bar{X}^j_s,\mathcal{L}_{\bar{X}^j_s}}}^2ds}\\
	  &\leq C t \int_0^t\e\zkuo{\abs{X^j_s - \bar{X}^j_s}^2 + \mathcal{W}_2\kuo{\frac{1}{n}\sum_{j=1}^n\delta_{X^j_s},\mathcal{L}_{\bar{X}^j_s}}^2}ds \\
	  &\qquad + C \int_0^t\e\zkuo{\abs{X^j_s - \bar{X}^j_s}^2 + \mathcal{W}_2\kuo{\frac{1}{n}\sum_{j=1}^n\delta_{X^j_s},\mathcal{L}_{\bar{X}^j_s}}^2}ds\\
	  & = C(1+t)\int_0^t\e\zkuo{\abs{X^j_s-\bar{X}^j_s}^2}ds + C(1+t)\int_0^t\e\zkuo{\mathcal{W}_2\kuo{\frac{1}{n}\sum_{j=1}^n\delta_{X^j_s},\mathcal{L}_{\bar{X}^j_s}}^2}ds\,.
\end{aligned}
\end{equation*}
We have 
$$
\sup _{s \leq t}\left|G_0\left(\law{U^{\xi}_s}\right)-G_0\left(\law{\tilde{U}_s}\right)\right| \leq \frac{M}{m} \sup _{s \leq t} \frac{1}{n} \sum_{j=1}^n\left|U_s^j-\bar{U}_s^j\right| \leq \frac{M}{m} \frac{1}{n} \sum_{j=1}^n \sup _{s \leq t}\left|U_s^j-\bar{U}_s^j\right|\,.
$$

Moreover, taking into account that the variables are exchangeable, the Cauchy-Schwarz inequality implies
$$
\mathbb{E}\left[\sup _{s \leq t}\left|G_0\left(\law{U^{\xi}_s}\right)-G_0\left(\law{\tilde{U}_s}\right)\right|^2\right] \leq \frac{M^2}{m^2} \frac{1}{n} \sum_{j=1}^n \mathbb{E}\left[\sup _{s \leq t}\left|U_s^j-\bar{U}_s^j\right|^2\right]=\frac{M^2}{m^2} \mathbb{E}\left[\sup _{s \leq t}\left|U_s^j-\bar{U}_s^j\right|^2\right]\,.
$$
Since
\begin{equation*}
\begin{aligned}
	  U^j_s - \bar{U}^j_s = \int_0^s b\kuo{X^j_r,\frac{1}{n}\sum_{j = 1}^n\delta_{X^j_r}} - b\kuo{\bar{X}^j_r,\mathcal{L}_{\bar{X}^j_r}}ds + \int_0^s \sigma\kuo{X^j_r,\frac{1}{n}\sum_{i = 1}^n\delta_{X^j_r}} - \sigma\kuo{\bar{X}^j_s,\mathcal{L}_{\bar{X}^j_s}}dB^j_s\,,
\end{aligned}
\end{equation*}
we have 
$$
\e\zkuo{\sup_{s\leq t} 
\abs{G_0(\law{U^{\xi}_s}) - G_0(\law{\bar{U}^j_s})}} 
\leq C\frac{M^2}{m^2}(1+t) \dkuo{\int_0^t\e\zkuo{\abs{X^j_s - \bar{X}^j_s}^2} + \e\zkuo{\mathcal{W}_2\kuo{\frac{1}{n}\sum_{j=1}^n\delta_{X^j_s},\mathcal{L}_{\bar{X}^j_s}}^2}ds}\,.
$$
Next, dealing with the $\e\zkuo{\mathcal{W}_2\kuo{\frac{1}{n}\sum_{i=1}^N\delta_{X^i_s},\mathcal{L}_{\bar{X}^j_s}}^2}$terms. 
By the triangle inequality, we get
\begin{equation*}
\begin{aligned}
\e\zkuo{\mathcal{W}_2\kuo{\frac{1}{n}\sum_{i=1}^n\delta_{X^i_s},\mathcal{L}_{\bar{X}^j_s}}}\leq \e\zkuo{\kuo{\frac{1}{n}\sum_{j = 1}^n\abs{X^j_s - \bar{X}^j_s}^2}^{\frac{1}{2}}+\mathcal{W}_2\kuo{\frac{1}{n}\sum_{j = 1}^n\delta_{\bar{X}^j_s},\mathcal{L}_{\bar{X}^j_s}}}\,.
\end{aligned}
\end{equation*}
Noting that the particles are exchangeable and from \cite{PTMFG}, we have 
$$
\e\zkuo{\mathcal{W}_2\kuo{\frac{1}{n}\sum_{j=1}^n\delta_{X^j_s},\mathcal{L}_{\bar{X}^j_s}}^2} \leq C \e\zkuo{\abs{X^j_s - \bar{X}^j_s}}   + C n^{-\frac{1}{2}}\,.
$$
Consequently, we have 
\begin{equation}
\begin{aligned}
	  \e\zkuo{\sup_{r\leq t}\abs{X^j_r - \bar{X}^j_r}^2}& \leq  2 A \int_0^t\e\zkuo{\sup_{r\leq s}\abs{X^j_r - \bar{X}^j_r}^2}ds\\
   &\qquad+ 4 \e\zkuo{\sup_{s\leq t}\abs{G_0(\law{{U}^{\xi}_s}) - G_0\kuo{\law{\tilde{U}_s}}}^2} + A T n^{-\frac{1}{2}},
\end{aligned}
\end{equation}
where $A = C(1+t) (1 + M^2/m^2)$. 
Apply Gronwall's lemma, we deduce that 
\begin{equation*}
\begin{aligned}
	  \e\zkuo{\sup_{r\leq t}\abs{X^j_r - \bar{X}^j_r}^2}\leq 2Ce^{A}
	  \dkuo{\e\zkuo{\sup_{s\leq t}\abs{G_0(\law{{U}^{\xi}_s}) - G_0\kuo{\law{\tilde{U}_s}}}^2} 
	  + AT n^{-\frac{1}{2}}}\,.
\end{aligned}
\end{equation*}
In view of Lemma \ref{lem:2.2}, we have
$$
\mathbb{E}\left[\sup _{s \leq t}\left|G_0\left(\law{{U}^{\xi}_s}\right)-G_0\left(\law{\tilde{U}_s}\right)\right|^2\right] \leq \frac{1}{m^2} \mathbb{E}\left[\sup _{s \leq t}\left|\int h\left(\bar{G}_0\left(\law{U^{\xi}_s}\right)+\cdot\right)\left(d \law{{U}^\xi_s}-d \law{\tilde{U}_s}\right)\right|^2\right]\,,
$$
which leads to
$$
\mathbb{E}\left[\sup _{r \leq t}\left|X_r^j-\bar{X}_r^j\right|^2\right] \leq C e^{A}\dkuo{ \frac{1}{m^2} \mathbb{E}\left[\sup _{s \leq t}\left|\int h\left(\bar{G}_0\left(\law{U^{\xi}_s}\right)+\cdot\right)\left(d \law{{U}^\xi_s}-d \law{\tilde{U}_s}\right)\right|^2\right] +ATn^{-\frac{1}{2}}}\,.
$$

As $h$ is Lipschitz continuous, the convergence rate is determined by the empirical measure convergence of independent and identically distributed diffusion processes. However, obtaining the standard convergence rate is not straightforward since we require uniform convergence in time. If we only suppose that Assumptions \ref{ass:A_1} and \ref{ass:A_2} holds, we obtain that
$$
\frac{1}{m^2} \mathbb{E}\left[\sup _{s \leq t}\left|\int h\left(\bar{G}_0\left(\law{\tilde{U}_s}\right)+\cdot\right)\left(d \law{{U}^{\xi}_s}-d \law{\tilde{U}_s}\right)\right|^2\right]\leq \frac{M^2}{m^2} \mathbb{E}\left[\sup _{s \leq t} \mathcal{W}_1\left(\law{{U}^{\xi}_s}, \law{\tilde{U}_s}\right)^2\right]\,.
$$

In view of the proof of \cite{briand2020particles}[Theorem 3.3 Part (i)]
, we have 
$$
\e\zkuo{\sup_{s\leq t}\mathcal{W}_1\kuo{\law{\bar{U}^N_s}, \law{U_s}}^2}\leq Cn^{-\frac{1}{2}}\,.
$$
So, we have there exists a constant $C>0$
\begin{equation*}
\begin{aligned}
	  \e\zkuo{\sup_{r\leq t}\abs{X^j_r - \bar{X}^j_r}^2}\leq Ce^A\kuo{C \frac{M^2}{m^2}+ AT}n^{-\frac{1}{2}}\leq C n^{-\frac{1}{2}}\,.
\end{aligned}
\end{equation*}
Then, we complete this proof.
\end{proof}

\section{Asymptotic analysis}\label{sec:5}

\subsection{Stability analysis}
In this subsection, We study stability properties of \eqref{eq:1.1} with respect to the initial condition,  coefficients, and driven process.

Firstly, we establish the stability of \eqref{eq:1.1} with respect to small perturbation of the initial point $\xi$. For simplicity, We denote by  $X^{\xi_n}_t$ as the unique solution  of following equation:
	$$
\left\{\begin{array}{l}
X^{\xi_n}_t=\xi_n+\int_0^t b\left(X^{\xi_n}_s, \mathcal{L}_{X^{\xi_n}_s}\right) d s+\int_0^t \sigma\left(X^{\xi_n}_s,\mathcal{L}_{X^{\xi_n}_s}\right) d B_s+K_t^{\xi_n}, \quad t \geq 0 \\
\mathbb{E}\left[h\left(X^{\xi_n}_t\right)\right] \geq 0, \quad \int_0^t \mathbb{E}\left[h\left(X^{\xi_n}_s\right)\right] d K_s^{\xi_n}=0, \quad t \geq 0
\end{array}\right.\,.
	$$
\begin{theorem}[Stability with respect to initial condition]\label{thm:sric}
Suppose that Assumptions \ref{ass:A_1} and \ref{ass:A_2} hold. The mapping 
$$
\Xi: \mathbb{R} \rightarrow L^2\kuo{\Omega, \mathcal{C}([0,T],\mathbb{R})}
$$
defined by $\Xi(\xi) \rightarrow X_t^\xi$ is continuous.
\end{theorem}

Secondly, we study the stability of \eqref{eq:1.1} with respect to small perturbation of the coefficients $b$ and $\sigma$. Let $X_t^n$ be the unique solution of the following equation:
$$
\left\{\begin{array}{l}
X^{n}_t=\xi+\int_0^t b^n\left(X^{n}_s, \mathcal{L}_{X^{n}_s}\right) d s+\int_0^t \sigma^n\left(X^{n}_s,\mathcal{L}_{X^{n}_s}\right) d B_s+K_t^{n}, \quad t \geq 0 \\
\mathbb{E}\left[h\left(X^{n}_t\right)\right] \geq 0, \quad \int_0^t \mathbb{E}\left[h\left(X^{n}_s\right)\right] d K_s^{n}=0, \quad t \geq 0
\end{array}\right.\,.
	$$
\begin{theorem}[Stability with respect to coefficients]\label{thm:src}
Assume that the functions $b(x,\mu)$, $b^n(x,\mu)$, $\sigma^n(x,\mu)$ and $\sigma(x,\mu)$ satisfy Assumptions \ref{ass:A_1} and \ref{ass:A_2}, Further suppose that for each $T>0$, and each compact set $K$ there exists $C>0$ such that
$$
{\lim_{n\rightarrow \infty} \sup _{x \in K}\sup_{\mu \in \mathcal{P}_2(\mathbb{R})} \abs{b^n(x,\mu) - b(x,\mu)} + \abs{\sigma^n(x,\mu) - \sigma(x,\mu)} = 0}\,,
$$
then 
$$
\lim_{n\rightarrow \infty} \e\zkuo{\sup_{t\leq T}\abs{X^n_t - X_t}^2} = 0\,.
$$
\end{theorem}

Thirdly, we study the stability with respect to the driving processes. We consider the MR-MVSDE driving by continuous semimartingales of the following form:
\begin{equation}\label{eq:s_DP_1}
\left\{\begin{array}{l}
X_t=\xi+\int_0^t b\left(X_s, \mathcal{L}_{X_s}\right) d A_s+\int_0^t \sigma\left(X_s,\mathcal{L}_{X_s}\right) d M_s+K_t, \quad t \geq 0 \\
\mathbb{E}\left[h\left(X_t\right)\right] \geq 0, \quad \int_0^t \mathbb{E}\left[h\left(X_s\right)\right] d K_s=0, \quad t \geq 0
\end{array}\right.\,,
\end{equation}
where $A_t$ is an adapted continuous process of bounded variation and $M_t$ is a continuous local martingale.

Let us consider the following sequence $\dkuo{X^n_t}_{n=1}^{\infty}$ of MR-MVSDE:
\begin{equation}\label{eq:s_DP_2}
\left\{\begin{array}{l}
X^n_t=\xi+\int_0^t b\left(X^n_s, \mathcal{L}_{X^n_s}\right) d A^n_s+\int_0^t \sigma\left(X^n_s,\mathcal{L}_{X^n_s}\right) d M^n_s+K^n_t, \quad t \geq 0 \\
\mathbb{E}\left[h\left(X^n_t\right)\right] \geq 0, \quad \int_0^t \mathbb{E}\left[h\left(X^n_s\right)\right] d K_s=0, \quad t \geq 0
\end{array}\right.\,,
\end{equation}
where $\left(A^n_t\right)$ is a sequence of $\mathcal{F}_t$-adapted continuous process of bounded variation and $M^n_t$ is continuous $\left(\mathcal{F}_t, \mathbb{P}\right)$-local martingales. {The proof of Theorem \ref{thm:1} can demonstrate that \eqref{eq:s_DP_2} is well-posedness. Detailed proof can be found in the appendix.}
\begin{theorem}[Stability with respect to driving processes]\label{thm:srdp}
Assume that the functions $b(x,\mu)$, $\sigma(x,\mu)$ satisfy Assumptions \ref{ass:A_1} and \ref{ass:A_2}, Further suppose that $(A_t,A_t^n,M_t,M_t^n)$ satisfy that
\begin{enumerate}[(i)]
	\item The family $\left(A_t, A_t^n, M_t, M_t^n\right)$ is bounded in $\mathcal{C}([0,T])^{\otimes 4}$,
	\item $\left(M^n_t-M_t\right)$ converges to 0 in probability in $\mathcal{C}([0,T])$,
	\item The total variation $\left(A^n_t-A_t\right)$ converges to 0 in probability.
\end{enumerate}
	Then
	$$
\lim _{n \rightarrow \infty} \e\left[\sup _{t<T}\left|X_t^n-X_t\right|^2\right]=0\,,
$$
where $(X^n_t )$ and $(X_t)$ are, respectively, solutions of \eqref{eq:s_DP_2} and \eqref{eq:s_DP_1}.
\end{theorem}

\subsubsection{Proofs of Theorem \ref{thm:sric}-\ref{thm:srdp}}
\begin{proof}[Proof of Theorem \ref{thm:sric}]

	Let $\{\xi_i\}_{i=1}^n$ be a sequence in $\mathbb{R}$ and $\lim_{i\rightarrow \infty} \xi_i = \xi$. We will show that 
	$$
	\lim_{n\rightarrow \infty} \e\zkuo{\sup_{t\leq T}\abs{X^{\xi_n}_t - X^\xi_t}^2} = 0.
	$$
	That means the mapping $\Xi$ is continuous.
	
	Apply the Lipschitz condition of $b$ and $\sigma$, and Cauchy-Schwarz inequality, we have 
	\begin{equation*}
\begin{aligned}
	  \e\zkuo{\sup_{t\leq T}\abs{X^{\xi_n}_t - X^{\xi}_t}^2} &\leq 4\abs{\xi_n - \xi}^2 + 4\e\zkuo{\sup_{t\leq T}\int_0^t\abs{b\kuo{X_s^{\xi_n},\mathcal{L}_{X_s^{\xi_n}}}-b\kuo{X_s^{\xi},\mathcal{L}_{X_s^{\xi}}}}ds}^2\\
	  &\qquad +4\e\zkuo{\sup_{t\leq T} \int_0^t \abs{\sigma\kuo{X_s^{\xi_n},\mathcal{L}_{X_s^{\xi_n}}}-\sigma\kuo{X_s^{\xi},\mathcal{L}_{X_s^{\xi}}}}dB_s}^2 \\&\qquad  + 4 \sup_{t\leq T}\abs{K^{\xi_n}_t - K^\xi_t}^2\\
	  &\leq 4\abs{\xi_n - \xi}^2 + 4 T\e\zkuo{\int_0^t \abs{b\kuo{X_s^{\xi_n},\mathcal{L}_{X_s^{\xi_n}}}-b\kuo{X_s^{\xi},\mathcal{L}_{X_s^{\xi}}}}^2ds}\\
	  &\qquad +4C\e\zkuo{\int_0^t \abs{\sigma\kuo{X_s^{\xi_n},\mathcal{L}_{X_s^{\xi_n}}}-\sigma\kuo{X_s^{\xi},\mathcal{L}_{X_s^{\xi}}}}^2ds}  
   \\
   &\qquad + 4 \sup_{t\leq T}\abs{K^{\xi_n}_t - K^\xi_t}^2\\
	  &\leq 4|\xi_n - \xi| + 4TC \e\zkuo{\int_0^t \abs{X^{\xi_n}_s - X^\xi_s}^2 + \mathcal{W}_2\kuo{\mathcal{L}_{X^{\xi_n}_s},\mathcal{L}_{X^\xi_s}}^2ds} \\
	  &\qquad+ 4C\e\zkuo{\int_0^t \abs{X^{\xi_n}_s - X^\xi_s}^2 + \mathcal{W}_2\kuo{\mathcal{L}_{X^{\xi_n}_s},\mathcal{L}_{X^\xi_s}}^2ds}  \\
   &\qquad + 4 \sup_{t\leq T}\abs{K^{\xi_n}_t - K^\xi_t}^2\,,
	  \end{aligned}
\end{equation*}
and 
\begin{equation*}
\begin{aligned}
	  	\sup_{t\leq T}\abs{K^{\xi_n}_t - K^\xi_t}^2&\leq \frac{M^2}{m^2}\e\zkuo{\sup_{r\in [0,t]}\abs{U^{\xi_n}_r - U_r^\xi}^2} \\
	  	&\leq \frac{M^2}{m^2}\dkuo{4|\xi_n - \xi| + (4T+1)C \e\zkuo{\int_0^t \abs{X^{\xi_n}_s - X^\xi_s}^2 + \mathcal{W}_2\kuo{\mathcal{L}_{X^{\xi_n}_s},\mathcal{L}_{X^\xi_s}}^2ds} }
\end{aligned}
\end{equation*}
Apply Gronwall's lemma, we have 
\begin{equation}
\begin{aligned}
	\e\zkuo{\sup_{t\leq T}\abs{X^{\xi_n}_t - X^{\xi}_t}^2} \leq 4\exp\dkuo{C(1+\frac{M^2}{m^2})}
	\end{aligned}|\xi_n - \xi|^2\,.
\end{equation}
Letting $n$ goes to infinity, we complete this proof.

\end{proof}
\begin{proof}[Proof of Theorem \ref{thm:src}]
	Apply the Lipschitz condition of $b^n$ and $\sigma^n$, and Cauchy-Schwarz inequality ,we have 
	\begin{equation*}
\begin{aligned}
	  \e\zkuo{\sup_{t\leq T}\abs{X^n_t - X_t}^2}&\leq 4\e\zkuo{\int_0^t b^n\kuo{X^n_s,\mathcal{L}_{X^n_s}}-b^n\kuo{X_s,\mathcal{L}_{X_s}}ds}^2\\
	  &\qquad +4\e\zkuo{\int_0^t b^n\kuo{X_s,\mathcal{L}_{X_s}}-b\kuo{X_s,\mathcal{L}_{X_s}}ds}^2\\
	  &\qquad +4\e\dkuo{\sup_{t\leq T}\zkuo{\int_0^t \sigma^n\kuo{X^n_s,\mathcal{L}_{X^n_s}}-\sigma^n\kuo{X_s,\mathcal{L}_{X_s}}dB_s}^2}\\
	  &\qquad +4\e\dkuo{\sup_{t\leq T}\zkuo{\int_0^t \sigma^n\kuo{X_s,\mathcal{L}_{X_s}}-\sigma\kuo{X_s,\mathcal{L}_{X_s}}dB_s}^2} +4\sup_{t\leq T}\abs{K^n_t - K_t}^2\\
	  &\leq 4(CT+C)\int_0^t\e\zkuo{|X^n_s-X_s|^2} + \mathcal{W}_2\kuo{\mathcal{L}_{X^n_s},\mathcal{L}_{X_s}}^2ds +C \lambda_n \\
   &\qquad+4\sup_{t\leq T}\abs{K^n_t - K_t}^2,
\end{aligned}
\end{equation*}
where 
$$
\lambda_n = \e\zkuo{\int_0^t \abs{b^n(X_s,\mathcal{L}_{X_s}) - b(X_s,\mathcal{L}_{X_s})}^2  + \abs{\sigma^n(X_s,\mathcal{L}_{X_s}) - \sigma(X_s,\mathcal{L}_{X_s})}^2ds}.
$$
Then, we control  $\sup_{t\leq T}\abs{K^{n}_t - K_t}^2$:
\begin{equation*}
\begin{aligned}
	  	\sup_{t\leq T}\abs{K^{n}_t - K_t}^2&\leq \frac{M^2}{m^2}\e\zkuo{\sup_{r\in [0,t]}\abs{U^{n}_r - U_r}^2} \\
	  	&\leq \frac{M^2}{m^2}4(CT+C)\int_0^t\e\zkuo{|X^n_s-X_s|^2} + \mathcal{W}_2\kuo{\mathcal{L}_{X^n_s},\mathcal{L}_{X_s}}^2ds +C \lambda_n \,.
	  	\end{aligned}
\end{equation*}
Apply Gronwall's lemma, we have 
\begin{equation*}
\begin{aligned}
\e\zkuo{\sup_{t\leq T}\abs{X^n_t - X_t}^2}\leq C(1+\frac{M^2}{m^2})\exp\dkuo{8C(1+T)(1+\frac{M^2}{m^2})}\lambda_n\,.
\end{aligned}
\end{equation*}
Letting $n$ go to infinity, we complete this proof.

\end{proof}
\begin{proof}[Proof of Theorem \ref{thm:srdp}]
	Apply the Lipschitz condition of $b$ and $\sigma$ and Cauchy-Schwarz inequality, we have 
	\begin{equation}\label{eq:th_5_3_1}
\begin{aligned}
	  \e\zkuo{\sup_{t\leq T}|X^n_t - X_t|^2}&\leq 4\e\dkuo{\sup_{t\leq T}\zkuo{\int_0^t \abs{b\kuo{X^n_s,\mathcal{L}_{X^n_s}}-b(X_s,\mathcal{L}_{X_s})}dA^n_s}^2} \\
	  &\qquad  + 4\e\dkuo{\sup_{t\leq T}\zkuo{\int_0^t \abs{\sigma\kuo{X^n_s,\mathcal{L}_{X^n_s}}-\sigma(X_s,\mathcal{L}_{X_s})}dM^n_s}^2}\\
	  &\qquad +4\e\zkuo{\sup_{t\leq T}\kuo{\int_0^t |b(X_s,\mathcal{L}_{X_s})|d|A^n_s-A_s|}^2} \\
	  &\qquad +4\e\zkuo{\sup_{t\leq T}\kuo{\int_0^t |\sigma(X_s,\mathcal{L}_{X_s})|d|M^n_s-M_s|}^2} +4\sup_{t\leq T}|K^n_t - K_t|^2\\
	  &\leq C\int_0^T\e\zkuo{\sup_{s\leq t}|X^n_s-X_s|^2}+\mathcal{W}_2\kuo{\mathcal{L}_{X^n_s},\mathcal{L}_{X_s}}^2 \kuo{dA^n_s + d\brackt{M^n}_s}+e_n\\
	  &\qquad +4\sup_{t\leq T}|K^n_t - K_t|^2\,,
\end{aligned}
\end{equation}
where
$$
e_n = 4\e\zkuo{\sup_{t\leq T}\kuo{\int_0^t |b(X_s,\mathcal{L}_{X_s})|d|A^n_s-A_s|}^2}+4\e\zkuo{\sup_{t\leq T}\kuo{\int_0^t |\sigma(X_s,\mathcal{L}_{X_s})|d|M^n_s-M_s|}^2}\,.
$$
Moreover, 
\begin{equation}\label{eq:th_5_3_2}
\begin{aligned}
	  	\sup_{t\leq T}\abs{K^{n}_t - K_t}^2&\leq \frac{M^2}{m^2}\e\zkuo{\sup_{r\in [0,t]}\abs{U^{n}_r - U_r}^2} \\
	  	&\leq \frac{M^2}{m^2}C\int_0^t\e\zkuo{|X^n_s-X_s|^2} + \mathcal{W}_2\kuo{\mathcal{L}_{X^n_s},\mathcal{L}_{X_s}}^2ds +C e_n 
	  	\end{aligned}
\end{equation}
 Combining \eqref{eq:th_5_3_1} and \eqref{eq:th_5_3_2}, we have 
\begin{equation*}
\begin{aligned}
  \e\zkuo{\sup_{t\leq T}|X^n_t - X_t|^2} \leq C \int_0^T\e\zkuo{\sup_{t\leq T}|X^n_t - X_t|^2}\kuo{dA^n_s + d\brackt{M^n}_s}+Ce_n\,.
\end{aligned}
\end{equation*}
Since $A^n_s+ \brackt{M^n}_s$ is an increasing process, then according to the Stochastic Gronwall lemma, we obtain
$$
 \e\zkuo{\sup_{t\leq T}|X^n_t - X_t|^2} \leq C e_n \e\zkuo{A^n_T+ \brackt{M^n}_T}<\infty,
$$ 
using the assumptions of $(A_t,A_t^n,M_t,M_t^n)$, we have 
$$
\lim_{n\rightarrow \infty} e_n = 0\,.
$$
Therefore,
$$
\lim\limits_{n\rightarrow \infty}\e\zkuo{\sup_{t\leq T}|X^n_t - X_t|^2} = 0\,.
$$
\end{proof}
\subsection{Large Deviation Principle}
For any $\varepsilon \in (0,1]$, let $X^{\varepsilon}=\left\{X_t^{\varepsilon}, t \in[0, T]\right\}$ be the unique strong solution to the following MR-MVSDE:
\begin{equation}\label{eq:4.1}
\left\{\begin{array}{l}
X^{\varepsilon}_t=\xi+\int_0^t b\left(X_s^{\varepsilon}, \mathcal{L}_{X_s^{\varepsilon}}\right) d s+\sqrt{\varepsilon}\int_0^t \sigma\left(X_s^{\varepsilon},\mathcal{L}_{X_s^{\varepsilon}}\right) d B_s+K_t^\varepsilon, \quad t \geq 0 \\
\mathbb{E}\left[h\left(X_t^{\varepsilon}\right)\right] \geq 0, \quad \int_0^t \mathbb{E}\left[h\left(X_s^{\varepsilon}\right)\right] d K_s^\varepsilon=0, \quad t \geq 0
\end{array}\right.\,.
\end{equation}
In this subsection, we mainly consider the LDP for $X^\varepsilon$ as $\varepsilon$ tending to 0. 

\begin{definition}[Large deviation]{}
 A family $\left\{X^{\varepsilon}\right\}_{\varepsilon>0}$ of $\mathcal{E}$-valued random variable is said to satisfy the large deviation principle on $\mathcal{E}$, with the good rate function $I$ and with the speed function $\lambda(\varepsilon)$ which is a sequence of positive numbers tending to $+\infty$ as $\varepsilon \rightarrow 0$, if the following conditions hold:
 \begin{enumerate}[(i)]
 	\item for each $M<\infty$, the level set $\{x \in \mathcal{E}: I(x) \leq M\}$ is a compact subset of $E$;
 	\item for each closed subset $F$ of $\mathcal{E}, \limsup _{\varepsilon \rightarrow 0} \frac{1}{\lambda(\varepsilon)} \log \mathbb{P}\left(X^{\varepsilon} \in F\right) \leq-\inf _{x \in F} I(x)$;
 	\item for each open subset $G$ of $\mathcal{E}, \liminf _{\varepsilon \rightarrow 0} \frac{1}{\lambda(\varepsilon)} \log \mathbb{P}\left(X^{\varepsilon} \in G\right) \geq-\inf _{x \in G} I(x)$.
 \end{enumerate}
\end{definition}

From Theorem \ref{thm:1}, we know that there exists a unique solution for the following oblique reflected  Ordinary differential equation (ODE for short)
\begin{equation}\label{eq:4.2}
\left\{\begin{array}{l}
\psi_t = \xi + \int_0^t b(\psi_s, \delta_{\psi_s})ds + K^\psi_t, \quad t \geq 0\\
h(\psi_t)\geq 0\quad \int_0^t h(\psi_s)dK^{\psi}_s = 0
\end{array}\right. ,
\end{equation}
where, 
$$
K^{\psi}_s = \sup_{r\leq s} \inf\dkuo{x\geq 0, h(x + U_s^{\psi_s}) \geq 0},
$$
and $\kuo{U^\psi}_{0\leq t\leq T}$ is the process defined by 
\begin{equation*}
\begin{aligned}
	  U^\psi_t = \xi + \int_0^tb\kuo{\psi_s, \delta_{\psi_s}}ds\,.
\end{aligned}
\end{equation*}

We denote $\psi$ as the solution of \eqref{eq:4.2}. For any $h \in L^2\kuo{[0,T];\mathbb{R}}$, consider the so-called skeleton equation:
\begin{equation}\label{eq:ske_eq}
\begin{aligned}
\left\{\begin{array}{l}
	  Y^h_t = \xi + \int_0^t b\kuo{Y^h_s,\delta_{\psi_s}}ds + \int_0^t \sigma\kuo{Y^h_s,\delta_{\psi_s}}h(s)ds + K^h_t\\
	  h(Y^h_t)\geq 0, \quad \int_0^t h(Y^h_s)dK^h_s = 0
	  \end{array}\right. \,.
\end{aligned}
\end{equation}
\begin{proposition}\label{pro:5.1}
	Under Assumptions \ref{ass:A_1} and \ref{ass:A_2},  there exists a unique strong solution to \eqref{eq:ske_eq}.
\end{proposition}
\begin{proof}
	{Let $\tilde{b}(y,t) =b\left( y, \delta_{\psi_t}\right)+\sigma\left(y, \delta_{\psi_t}\right) h(t)$. 
	Since $b$ and $\sigma$ satisfy assumption \ref{ass:A_1}, and $h \in L^2([0,T];\mathbb{R})$, we can easily prove the existence and uniqueness of solution for \ref{eq:ske_eq} through the standard argument (Picard iteration).
	 it is easy show \eqref{eq:ske_eq} has unique solution.
	 } 
\end{proof}
\begin{theorem}[Large deviation principle]\label{thm:LDP}
Let Assumptions \ref{ass:A_1} and \ref{ass:A_2} hold and $X^{\varepsilon}$ be the unique strong solution to \eqref{eq:4.1}. Then the family of $\left\{X^{\varepsilon}\right\}_{\varepsilon>0}$ satisfies a LDP on the space $\mathcal{X}$ with rate function
$$
I(\phi):=\inf _{\left\{h \in L^2\left([0, T], \mathbb{R}\right): \phi=Y^h\right\}} \frac{1}{2} \int_0^T|h(t)|^2 \mathrm{~d} t, \quad \phi \in \mathcal{X}
$$
with the convention $\inf \{\emptyset\}=+\infty$, here $Y^h \in \mathcal{X}$ solves equation \eqref{eq:ske_eq}.
\end{theorem}
\begin{proof}
	According to Proposition \ref{pro:5.1} , there exists a measurable map $\mathcal{G}^0$: 
	$$
	\mathcal{G}^0 : C\left([0, T] ; \mathbb{R}\right) \rightarrow \mathcal{X} \,,
	$$
	such that for any $h \in L^2([0,T],\mathbb{R})$
	$$
	Y^h = \mathcal{G}^0\kuo{\int_0^\cdot h(s)ds}.
	$$
	Let $$
\mathcal{H}^M:=\left\{h:[0, T] \rightarrow \mathbb{R}^m, \int_0^T\|h(s)\|^2 \mathrm{~d} s \leq M\right\}\,,
$$
and
$$\tilde{\mathcal{H}}^M:=\left\{h: h\right. \text{is }\mathbb{R}\text{-valued }\mathcal{F}_t\text{-predictable process such that } h(\omega) \in \mathcal{H}^M, \mathbb{P}-a.s \}\,.$$

$\mathcal{H}^M$ is endowed with the weak topology on $L^2\left([0, T], \mathbb{R}\right)$. Then $\mathcal{H}^M$ is a Polish space. For every $\varepsilon > 0$, there exists a measurable mapping $\mathcal{G}^\varepsilon(\cdot) :C\left([0, T] ; \mathbb{R}\right) \rightarrow \mathcal{X} $ such that 
$$
X^\varepsilon = \mathcal{G}^\varepsilon(B(\cdot))\,,
$$
and for any $M > 0$ and $h^\varepsilon \in \tilde{\mathcal{H}}^M$ , 
$$
Z^{\varepsilon}:=\mathcal{G}^{\varepsilon}\left(B(\cdot)+\frac{1}{\sqrt{\varepsilon}} \int_0^{\cdot} h^{\varepsilon}(s) \mathrm{d} s\right)\,.
$$
From  Girsanov's Theorem, we know $Z^\varepsilon$  is the solution of the following MR-MVSDE:
\begin{equation}\label{eq:Z}
\begin{aligned}
	  \left\{\begin{array}{l}
	  Z^\varepsilon_t = \xi + \int_0^t b\kuo{Z^\varepsilon_s,\mathcal{L}_{X^\varepsilon_s}}ds + \int_0^t \sigma\kuo{Y^h_s,\mathcal{L}_{X^\varepsilon_s}}h^\varepsilon(s)ds + K^{Z^\varepsilon}_t\\
	  \e\zkuo{h(Z^\varepsilon_t)}\geq 0, \quad \int_0^t h(Z^\varepsilon_s)dK^{Z^\varepsilon}_s = 0
	  \end{array}\right. \,.
\end{aligned}
\end{equation}
Following \cite{liu2022large}, we only need to verify the two claims:
\begin{itemize}
	\item [\textbf{(LDP1)}]  For every $M<+\infty$ and any family $\left\{h^n ; n \in \mathbb{N}\right\} \subset \mathcal{H}^M$ converging to some element $h \in \mathcal{H}^M$ as $n \rightarrow \infty$
	$$
	\lim\limits_{n\rightarrow \infty} \sup_{t\in [0,T]}\abs{\mathcal{G}^0\kuo{\int_0^{\cdot}h^n(s)ds} - \mathcal{G}^0\kuo{\int_0^{\cdot}h(s)ds}} = 0\,.
	$$
	\item [\textbf{(LDP2)}] For every $M<+\infty$ and any family $\left\{h^{\varepsilon} ; \varepsilon>0\right\} \subset \tilde{\mathcal{H}}^M$ and any $\theta>0$,
$$
	\lim_{\varepsilon\rightarrow 0}\p\zkuo{\sup_{t\in [0,T]}\abs{Z^\varepsilon_t - Y^{h^\varepsilon}_t} \geq \theta} = 0\,,
	$$
where $Z^{\varepsilon}$ is the unique solution to \eqref{eq:Z} and $Y^{h^{\varepsilon}}=\mathcal{G}^0\left(\int_0^\cdot h^{\varepsilon}(s) \mathrm{d} s\right)$.
\end{itemize}

The verification of (LDP1) will be given in Proposition \ref{pro:LDP1}. (LDP2) will be established in Proposition \ref{pro:LDP2}.
\end{proof}
\subsubsection{Proof of (LDP1)}
\begin{proposition}\label{pro:LDP1}
Let Assumptions \ref{ass:A_1} and \ref{ass:A_2} hold. For any $M< + \infty$, family $\{h^n\}_{n\geq 1}  \in \mathcal{H}^M$ and $h \in \mathcal{H}^M$, suppose that $h^n \rightarrow h$ as $n \rightarrow \infty$ in the weak topology of $L^2([0,T],\mathbb{R})$. Then
	$$
	\lim\limits_{n\rightarrow \infty} \sup_{t\in [0,T]}\abs{\mathcal{G}^0\kuo{\int_0^{\cdot}h^n(s)ds} - \mathcal{G}^0\kuo{\int_0^{\cdot}h(s)ds}} = 0\,.
	$$
\end{proposition}
\begin{proof}
	For simplicity, we denote $Y^{h_n} : = \mathcal{G}^0\kuo{\int_0^{\cdot}h^n(s)ds}$, $Y^h = \mathcal{G}^0\kuo{\int_0^{\cdot}h(s)ds}$.
	\begin{equation}\label{eq:4_pro_1_1}
\begin{aligned}
	  Y^{h_n} - Y^h &= \int_0^t b\kuo{Y^{h^n}_s,\delta_{\psi_s}} - b\kuo{Y^h_s,\delta_{\psi_s}}ds + \int_0^t\sigma\kuo{Y^{h^n}_s,\delta_{\psi_s}} h^n(s) - \sigma\kuo{Y^h_s,\delta_{\psi_s}}h(s)ds \\
	  &\qquad+ K^{h^n}_t - K^h_t\\
	  & = \int_0^t b\kuo{Y^{h^n}_s,\delta_{\psi_s}} - b\kuo{Y^h_s,\delta_{\psi_s}}ds + \int_0^t\zkuo{\sigma\kuo{Y^{h^n}_s,\delta_{\psi_s}} - \sigma\kuo{Y^h_s,\delta_{\psi_s}}}h^n(s)ds\\
	  &\qquad + \int_0^t\sigma\kuo{Y^h_s,\delta_{\psi_s}}\zkuo{h^n(s)-h(s)}ds + K^{h^n}_t - K^h_t\\
	  &:= I^n_1(t) + I^n_2(t) + I^n_3(t) + I^n_4(t).
\end{aligned}
\end{equation}
Let $\kappa^n(t):= \sup_{r\in [0,t]}\abs{Y^{h^n}(r) - Y^h(r)}$. By the Lipschitz condition of $b$, we have 
\begin{equation}\label{eq:4_pro_1_2}
\begin{aligned}
	  \abs{I^n_1(t)} \leq \int_0^t \abs{b\kuo{Y^{h^n}_s ,\delta_{\psi_s}}-b\kuo{Y^h_s,\delta_{\psi_s}}}ds \leq K\int_0^t \kappa^n(s)ds.
\end{aligned}
\end{equation}
By the Lipschitz condition of $\sigma$, we have 
\begin{equation}\label{eq:4_pro_1_3}
\begin{aligned}
	  \abs{I^n_2(t)}\leq K\int_0^t|h^n(s)|\kappa_n(s)ds\,.
\end{aligned}
\end{equation}
By the linear growth of $\sigma$, we have 
\begin{equation*}
\begin{aligned}
	  \sup_{t\in [0,T]}\abs{\sigma(x,\mu)}^2\leq C(1+K^2)
\end{aligned}
\end{equation*}
Since $h^n(t) \rightarrow h(t)$ weakly in $L^2\kuo{[0,T],\mathbb{R}}$ as $n\rightarrow \infty$, we have 
\begin{equation*}
\begin{aligned}
	  \int_0^T\abs{h^n(s)}^2ds \leq C(h) \quad \int_0^T\abs{h(s)}^2ds \leq C(h).
\end{aligned}
\end{equation*}
Then, by the Cauchy-Schwarz inequality, we have 
\begin{equation}\label{eq:4_pro_1_4}
\begin{aligned}
	  \sup_{t\in [0,T]} \abs{I^n_3(t)}&\leq \int_0^T \abs{\sigma\kuo{Y^h_s,\delta_{\psi_s}}\zkuo{h^n(s)-h(s)}}ds\\
	  &\leq \kuo{\int_0^T\abs{\sigma\kuo{Y^h_s,\psi_s}}^2ds}^{\frac{1}{2}}\kuo{\int_0^T\abs{h^n(s)-h(s)}^2}^{\frac{1}{2}}\\
	  &\leq \kuo{ T \sup_{t\in [0,T]}\abs{\sigma\kuo{Y^h_s,\psi_s} }^2}^{\frac{1}{2}}\kuo{2 \int_0^T\kuo{|h^n(s)|^2+|h(s)|^2}ds}^{\frac{1}{2}}\\
	  &\leq (CT(1+K^2)C(h))^{\frac{1}{2}}<\infty .
\end{aligned}
\end{equation}
Similarly, we have, for any $0\leq t_1\leq t_2\leq T$,
$$
\abs{I^n_3(t_2) - I^n_3(t_1)}\leq \zkuo{C(t_2-t_1)(1+K^2)C(h)}^{\frac{1}{2}},
$$
which means that the sequence $\dkuo{I^n_3:n\geq 1}$ is equi-continuous. By the Arzela-Ascoli theorem, we deduce that $\dkuo{I^n_3:n\geq 1}$ is relatively compact in $\mathcal{C}\kuo{[0,T],\mathbb{R}}$.

We can deduce that for any $t \in [0,T]$,
\begin{equation}\label{eq:4_pro_1_5}
\begin{aligned}
	  I^n_3(t) = \int_0^t\kuo{h^n(s)-h(s)}\sigma\kuo{Y^h_s,\delta_{\psi_s}}ds\rightarrow 0 \quad \text{as } n \rightarrow \infty\,.
\end{aligned}
\end{equation}
Together with the relative compactness of $\dkuo{I^n_3:n\geq 1}$ in $\mathcal{C}\kuo{[0,T],\mathbb{R}}$, we can get 
\begin{equation*}
\begin{aligned}
	  \lim\limits_{n\rightarrow \infty}\sup_{t\in [0,T]}|I^n_3(t)| = 0\,.
\end{aligned}
\end{equation*}
By Lemma \ref{lem:2.2} , we have 
\begin{equation}\label{eq:4_pro_1_6}
\begin{aligned}
	  K^{h^n}_t - K^h_t\leq \frac{M}{m}\sup_{r\in [0,t]}\abs{U^{h^n}_r - U^h_r} = \frac{M}{m}\sup_{r \in [0,t]}\zkuo{|I^n_1|(r) + |I^n_2|(r)}\,.
	  \end{aligned}
\end{equation}
Combing \eqref{eq:4_pro_1_1}, \eqref{eq:4_pro_1_2}, \eqref{eq:4_pro_1_3}, \eqref{eq:4_pro_1_4} and \eqref{eq:4_pro_1_6} together, we have 
\begin{equation*}
\begin{aligned}
	  \kappa^n(t)&\leq C\dkuo{\int_0^t\kappa^n(s)ds + \int_0^t\kappa^n(s)\abs{h^n(s)}ds}+ 2\sup_{t\in [0,T]}\abs{I^n_3(t)}\,.
\end{aligned}
\end{equation*}
Then, by \eqref{eq:4_pro_1_5} and Gronwall's lemma, we have 
\begin{equation}
\begin{aligned}
	  \kappa^n(T)&\leq 2\sup_{t\in [0,T]}|I^n_3(t)|\exp\dkuo{C\int_0^T|h^n(s)|ds}\\
	  & \leq C\sup_{t\in [0,T]}|I^n_3(t)| \rightarrow 0 \quad \text{as }n\rightarrow\infty.
\end{aligned}
\end{equation}
The proof is complete.
\end{proof}
\subsubsection{Proof of (LDP2)}
\begin{lemma}\label{lem:4.4}
Suppose Assumptions \ref{ass:A_1} and \ref{ass:A_2} hold, let $X^\varepsilon_t$ and $\psi_t$ satisfy \eqref{eq:4.1} and \eqref{eq:4.2}, respectively. Then
	$$
	\lim_{\varepsilon \rightarrow 0}\sup_{t \in [0,T]}\mathcal{W}_2\kuo{\mathcal{L}_{X^\varepsilon_t},\delta_{\psi_t}}^2 = 0\,.
	$$
\end{lemma}
\begin{proof}
	\begin{equation}\label{eq:4_lemma_1_1}
\begin{aligned}
	  \e\zkuo{\sup_{t\in [0,T]}\abs{X^\varepsilon_t - \psi_t}^2}&\leq 3\left\{\e\zkuo{ \sup_{t\in [0,T]}\int_0^tb\kuo{X^\varepsilon_s,\mathcal{L}_{X^\varepsilon_s}}-b\kuo{\psi_s,\delta_{\psi_s}}ds}^2
	  \right.\\
	  &\qquad \left.+\varepsilon\e\zkuo{\sup_{t\in [0,T]}\int_0^t\sigma\kuo{X^\varepsilon_s,\mathcal{L}_{X^\varepsilon_s}}dB_s}^2 + \sup_{t\in [0,T]}\abs{K^\varepsilon_t - K^\psi_t}^2\right\}\\
	  &=: I_1 + I_2 + I_3.
\end{aligned}
\end{equation}

By the Lipschitz condition of $b$ and $\sigma$, we have 
\begin{equation}\label{eq:4_lemma_1_2}
\begin{aligned}
	  I_1\leq TK\int_0^T\dkuo{\e\zkuo{\abs{X^\varepsilon_s - \psi_s}^2}+ \mathcal{W}_2\kuo{\mathcal{L}_{X^\varepsilon_s},\delta_{\psi_s}}^2}ds\,,
\end{aligned}
\end{equation}
\begin{equation}\label{eq:4_lemma_1_3}
\begin{aligned}
	  {I_2\leq \e\int_0^T \sigma\kuo{X^\varepsilon_s,\mathcal{L}_{X^\varepsilon_s}}^2ds\leq \e\int_0^TC(1+{X^\varepsilon_s})^2ds}\,,
\end{aligned}
\end{equation}
\begin{equation}\label{eq:4_lemma_1_4}
\begin{aligned}
	  I_3 = \sup_{t\in [0,T]}\abs{K^\varepsilon_t - K^\psi_t}^2&\leq\frac{M^2}{m^2}\e\zkuo{\sup_{r\in [0,t]}\abs{U^\varepsilon_r- U^\psi_r}^2}\leq 3 \frac{M^2}{m^2}(I_1+I_2)\,.
\end{aligned}
\end{equation}
Putting \eqref{eq:4_lemma_1_2}, \eqref{eq:4_lemma_1_3} and \eqref{eq:4_lemma_1_4} to \eqref{eq:4_lemma_1_1} and applying the Gronwall's lemma, we have 
\begin{equation}\label{eq:4_lemma_1_5}
\begin{aligned}
	  \e\zkuo{\sup_{t\in [0,T]}\abs{X^\varepsilon_t - \psi_t}^2}\leq e^{C}\dkuo{T \int_0^T\mathcal{W}_2\kuo{\mathcal{L}_{X^\varepsilon_s},\delta_{\psi_s}}^2ds + \varepsilon}\,.
\end{aligned}
\end{equation}

By the definition of $\mathcal{W}_2\kuo{\mathcal{L}_{X^\varepsilon_t},\delta_{\psi_t}}$, we have following estimation
\begin{equation*}
\begin{aligned}
\mathcal{W}_2\kuo{\mathcal{L}_{X^\varepsilon_t},\delta_{\psi_t}}^2\leq \e\zkuo{\sup_{t\in [0,T]}\abs{X^\varepsilon_t - \psi_t}^2}\leq e^{C}\dkuo{T \int_0^T\mathcal{W}_2\kuo{\mathcal{L}_{X^\varepsilon_s},\delta_{\psi_s}}^2ds + \varepsilon}\,.
\end{aligned}
\end{equation*}
Thanking to  Gronwall's lemma, we have 
$$
	\lim_{\varepsilon \rightarrow 0}\sup_{t \in [0,T]}\mathcal{W}_2\kuo{\mathcal{L}_{X^\varepsilon_t},\delta_{\psi_t}}^2 = 0\,.
	$$

\end{proof}
\begin{proposition}\label{pro:LDP2}
Let Assumptions \ref{ass:A_1} and \ref{ass:A_2} hold. Then for any $\theta > 0$,
	$$
	\lim_{\varepsilon\rightarrow 0}\p\zkuo{\sup_{t\in [0,T]}\abs{Z^\varepsilon_t - Y^{h^\varepsilon}_t} \geq \theta} = 0\,.
	$$
\end{proposition}
\begin{proof}
By Chebyshev's inequality, we only need to show that 
$$
\lim _{\varepsilon \rightarrow 0} \mathbb{E}\left[\sup _{t \in[0, T]}|Z_t^{\varepsilon}-Y_t^{h^{\varepsilon}}|^2\right]=0\,.
$$
Then, it follows that 
\begin{equation}\label{eq:4_pro_2_1}
\begin{aligned}
	  \e\zkuo{\sup_{t\in [0,T]}\abs{Z^\varepsilon_t - Y^{h^\varepsilon_t}}^2}&\leq T\e\dkuo{\sup_{t\in [0,T]}\int_0^T\abs{b\kuo{Z^\varepsilon_s,\mathcal{L}_{X^\varepsilon_s}}-b\kuo{Y^{h^\varepsilon}_s,\delta_{\psi_s}}}^2ds}\\
	  &\qquad + \e\dkuo{\sup_{t\in [0,T]}\zkuo{\int_0^T 
	  h^\varepsilon(s)
	  \zkuo{\sigma \kuo{Z^\varepsilon_s,\mathcal{L}_{X^\varepsilon_s}} - \sigma\kuo{Y^{h^\varepsilon}_s,\delta_{\psi_s}}}ds}^2}\\
	  &\qquad + \varepsilon C_p\int_0^T\sigma\kuo{Z^\varepsilon_s,\mathcal{L}_{X^\varepsilon_s}}^2ds + \sup_{t\in [0,T]}\abs{K^{Z^\varepsilon}_t - K^{h^\varepsilon}_t}^2\\
	  &\leq KT\sup_{t\in [0,T]}\e\int_0^T\zkuo{\abs{Z^\varepsilon_s - Y^{h^\varepsilon}_s}^2 + \mathcal{W}_2\kuo{\mathcal{L}_{X^\varepsilon_s},\delta_{\psi_s}}^2}ds \\
	  &\qquad+N\sup_{t\in [0,T]}\e\int_0^T\zkuo{\abs{Z^\varepsilon_s - Y^{h^\varepsilon}_s}^2 + \mathcal{W}_2\kuo{\mathcal{L}_{X^\varepsilon_s ,\delta_{\psi_s}}}^2}ds \\
	  &\qquad +\sup_{t\in [0,T]}\abs{K^{Z^\varepsilon}_t - K^{h^\varepsilon}_t}^2\,.
	  \end{aligned}
\end{equation}
And 
\begin{equation}\label{eq:4_pro_2_2}
\begin{aligned}
	  \sup_{t\in [0,T]}\abs{K^{Z^\varepsilon}_t - K^{h^\varepsilon}_t}^2 &\leq \frac{M^2}{m^2}\sup_{r\in [0,t]}\abs{U^{Z^\varepsilon}_r - U^{h^{\varepsilon}}_r}^2 \\
	  &\leq \frac{M^2}{m^2}KT\int_0^T\e\zkuo{\sup_{t\in [0,T]}\abs{Z^\varepsilon_s - Y^{h^\varepsilon}_s}^2 + \mathcal{W}_2\kuo{\mathcal{L}_{X^\varepsilon_s},\delta_{\psi_s}}^2}ds \\
	  &\qquad +\frac{M^2}{m^2}N\int_0^T\e\zkuo{\sup_{t\in [0,T]}\abs{Z^\varepsilon_s - Y^{h^\varepsilon}_s}^2 + \mathcal{W}_2\kuo{\mathcal{L}_{X^\varepsilon_s ,\delta_{\psi_s}}}^2}ds\,.
\end{aligned}
\end{equation}
By \eqref{eq:4_pro_2_1} and \eqref{eq:4_pro_2_2} and apply Gronwall's lemma, we deduce that 
\begin{equation}
\begin{aligned}
	  \e\zkuo{\sup_{t\in [0,T]}\abs{Z^\varepsilon_t - Y^{h^\varepsilon_t}}^2}&\leq \kuo{CT+N}(1+\frac{M^2}{m^2})\int_0^t\e\sup_{t\in[0,s]}\abs{Z^\varepsilon_s - Y^{h^\varepsilon}_s}^2ds \\
	  &\qquad+ \kuo{CT+N}(1+\frac{M^2}{m^2}) T\sup_{t\in [0,T]}\mathcal{W}\kuo{\mathcal{L}_{X^\varepsilon_t},\delta_{\psi_t}} + C(1+\frac{M^2}{m^2})\varepsilon\,.
	  \end{aligned}
\end{equation}

Apply Gronwall's lemma, Lemma \ref{lem:4.4} and letting $\varepsilon \rightarrow \infty$, we get 
$$
\lim _{\varepsilon \rightarrow 0} \mathbb{E}\left[\sup _{t \in[0, T]}|Z_t^{\varepsilon}-Y_t^{h^{\varepsilon}}|^2\right]=0\,.
$$
The proof is complete.
\end{proof}
\section{Invariant probability measure and Harnack inequality}\label{sec:6}

\subsection{Existence of invariant probability measure}
For any $X_{s,t}^x$ is the solution of \eqref{eq:1.1} with initial condition $x$ from $s$ to $t$,  let us define operator $P_{s, t}$ and its adjoin $P^\star_{s,t}$ by
$$
P_{s, t} f(x):=\mathbb{E} f\left(X_{s, t}^x\right)=\int_{\mathbb{R}} f(y)\left(P_{s, t}^* \delta_x\right)(\mathrm{d} y)\,.
$$
$P_{s, t}$ is not a semigroup, but the dual  $P_{s, t}^* $ is a non-linear semigroup on $\mathcal{P}_2(\mathbb{R})$.

We first prove the exponential convergence.
\begin{lemma}[Exponential convergence]\label{lem:6.1}
	Suppose Assumptions \ref{ass:A_1}, \ref{ass:A_2} and \ref{ass:A_5} hold, let $X_t$ and $Y_t$ be two solution of \eqref{eq:1.1} such that $\mathcal{L}_{\xi^{X_0}} = \mu_0$, $\mathcal{L}_{\xi^{Y_0}} = \nu_0$ and $\mathcal{W}_2\kuo{\mu_0,\nu_0}^2 = \e\zkuo{\abs{\xi^{X_0}-\xi^{Y_0}}^2}$. We have 
	$$
	\mathcal{W}_2\kuo{P^*_{t}\mu_0,P^*_t\nu_0}^2\leq \mathrm{e}^{-\left(C_2-C_1\right) t} \mathcal{W}_2\left(\mu_0, \nu_0\right)^2\,.
	$$
\end{lemma}
\begin{proof}
	From the definition of $P^*_t$, we have $P^*_t\mu_0 = \mu_t$, $P^*_t\nu_0 = \nu_t$. 
	
	Apply It\^{o}'s lemma to $(X_t- Y_t)^2$, we have 
	\begin{equation}\label{eq:lem_6_1_1}
\begin{aligned}
	  \kuo{X_t - Y_t}^2 &= (\xi^{X_0}-\xi^{Y_0})^2+\int_0^t 2\kuo{X_t - Y_t}\zkuo{b(X_t,\mu_t) - b(Y_t,\nu_t)}dt + \int_0^t \sigma(X_t,\mu_t) - \sigma(Y_t,\nu_t)dt \\
	  &\qquad+ \int_0^t 2\kuo{X_t - Y_t}\zkuo{\sigma(X_t,\mu_t) - \sigma(Y_t,\nu_t)}dB_t + \int_0^t 2(X_t-Y_t)d(K^X_t - K^Y_t)\,.
	  \end{aligned}
\end{equation}
Since $X_t$ and $Y_t$ is the solution of \eqref{eq:1.1} and $h$ is bi-Lipschitz , we have $\e\zkuo{X_t}\geq 0$ and $\e\zkuo{Y_t}\geq 0$. Combining  $K^X_t$ and $K^Y_t$ are non-decrease processes,  we have 
$$
\int_0^t\e\zkuo{X_t}dK^Y_t\geq 0\,, \quad \int_0^t\e\zkuo{Y_t}dK^X_t \geq 0\,.
$$
By the definition of solution \eqref{eq:1.1}, we obtain $$
\int_0^t\e\zkuo{h(X_t)}dK^X_t =\int_0^t \e\zkuo{h(Y_t)}dK^Y_t = 0\,.
$$
Taking expectation of \eqref{eq:lem_6_1_1} and applying stochastic Fubini's Theorem, we have the following estimation 
$$
\e\kuo{X_t - Y_t}^2\leq  (\xi^{X_0}-\xi^{Y_0})^2+\e\int_0^t 2\kuo{X_t - Y_t}\zkuo{b(X_t,\mu_t) - b(Y_t,\nu_t)}dt + \e\int_0^t \sigma(X_t,\mu_t) - \sigma(Y_t,\nu_t)dt $$
Apply Assumption \ref{ass:A_5}  and Gronwall's lemma, we obtain
$$
\mathcal{W}_2\kuo{P^*_{t}\mu_0,P^*_t\nu_0}^2 \leq \e\zkuo{|X_t-Y_t|^2}\leq e^{-(C_2-C_1)t}\e\zkuo{|\xi^{X_0}-\xi^{Y_0}|^2}\leq  e^{-(C_2-C_1)t} \mathcal{W}_2\left(\mu_0, \nu_0\right)^2\,.
$$
\end{proof}
\begin{theorem}[Existence of invariant probability measure]\label{thm:eipm}
Suppose Assumptions \ref{ass:A_1}, \ref{ass:A_2} and \ref{ass:A_5} hold. $P_t$ has a unique invariant probability measure $\mu \in \mathcal{P}_2$ i.e. $P_t^* \mu=\mu, t>0$ such that 
$$
\mathcal{W}_2\left(P_t^* \nu_0, \mu\right)^2 \leq \mathcal{W}_2\left(\nu_0, \mu\right)^2 \mathrm{e}^{-\left(C_2-C_1\right) t}, \quad t \geq 0, \nu_0 \in \mathcal{P}_2 .
$$
\end{theorem}

\begin{proof}
	{We first show that there exists a $\mu \in \mathcal{P}_2(\mathbb{R})$} such that 
	\begin{equation}\label{eq:thm_6_2_1}
\begin{aligned}
	  	\lim_{t\rightarrow \infty}\mathcal{W}_2\kuo{P^*_t\delta_0,\mu} = 0\,.
\end{aligned}
\end{equation}
	To this end, it suffices to show that $\left\{P_t^* \delta_0\right\}_{t \geq 0}$ is a $\mathcal{W}_{2}$-Cauchy family when $t \rightarrow \infty$. Using the semigroup property of  $P_{t+s}^*$ and the exponential convergence, we have
$$
\begin{aligned}
 \mathcal{W}_2\left(P_{t+s}^* \delta_0, P_t^* \delta_0\right)^2 \leq \mathcal{W}_2\left(P_s^* \delta_0, \delta_0\right)^2 \mathrm{e}^{-\left(C_2-C_1\right) t}  =\mathrm{e}^{-\left(C_2-C_1\right) t} \mathbb{E}\left|X_s^0\right|^2, \quad s, t \geq 0 \,,
\end{aligned}
$$
where $X^0_s$ is a solution of MR-MVSDE whose initial condition is $\xi = 0$. Then, we only need to show 
$$
\sup_{s\geq 0}\e\zkuo{|X^0_s|^2}< \infty\,.
$$
It is true from Proposition \ref{pro:3.2}. By \eqref{eq:thm_6_2_1}, $P^*_{t+s} = P^*_t P^*_s$ and 
$$
\mathcal{W}_2\kuo{P^*_t\mu_1,P^*_t \mu_2}\leq e^{-(C_2-C_1)t/2}\mathcal{W}_2\kuo{\mu_1,\mu_2}\,,
$$
we obtain that
$$
\mathcal{W}_2\kuo{P^*_t\mu,\mu} = \lim_{s\rightarrow \infty}\mathcal{W}_2\kuo{P^*_t \mu , P^*_{t+s}\delta_0}\leq e^{-(C_2-C_1)t/2}\lim_{s\rightarrow \infty}\mathcal{W}_2\kuo{\mu,P^*_s\delta_0} = 0\,.
$$
That means $P^*_t \mu = \mu$.
\end{proof}
\subsection{Log-Harnack inequality and Shift Harnack inequality}
In this subsection, we investigate the dimension-free log-Harnack inequality and Shift Harnack inequality. Using the coupling by change of measure method, we establish log-Harnack inequalities for $P_t f$, as detailed in \cite{WFY_Harnack_AOP,wang2013harnack}. However, for this purpose, we need to assume that the noise term is distribution-free and that $|\sigma(x)^{-1}|\leq \lambda$. This implies that we consider a particular version of \eqref{eq:1.1}:
\begin{equation}\label{eq:6_2_1}
\left\{\begin{array}{l}
X_t=\xi^{X_0}+\int_0^t b\left(X_s, \mathcal{L}_{X_s}\right) d s+\int_0^t \sigma\left(X_s\right) d B_s+K_t, \quad t \geq 0 \\
\mathbb{E}\left[h\left(X_t\right)\right] \geq 0, \quad \int_0^t \mathbb{E}\left[h\left(X_s\right)\right] d K_s=0, \quad t \geq 0
\end{array}\right.\,.
\end{equation}
	For any $\mu_0 \in \mathcal{P}_2(\mathbb{R})$ and $r \geq 0$, let $B(\mu_0,r) = \dkuo{\nu\in \mathcal{P}_2(\mathbb{R}):\mathcal{W}_2(\mu_0,\nu)\leq r}$. Let 
	$$
	\phi(s,t) = \lambda^2 \kuo{\frac{C_1}{1-e^{-C_1(t-s)}}+\frac{tC_2^2\exp\kuo{2(C_1+C_2)(t-s)}}{2}},\quad 0\leq s<t\,.
	$$
Let $\xi_t = C_1^{-1}(1-e^{C_1(t- T)})$, $\nu_t = P^*_t\nu_0$ and let $\xi^{Y_0}$ be $\mathcal{F}_0-$measurable with $\mathcal{L}_{\xi^{Y_0}} = \nu_0$. Consider the following MR-MVSDE:
\begin{equation}\label{eq:6_2_2}
\left\{\begin{array}{l}
Y_t=\xi^{Y_0}+\int_0^t b\left(Y_s, \mathcal{L}_{Y_s}\right) d s+\int_0^t \frac{1}{\xi_s}\sigma(Y_s) \sigma(X_s)^{-1}(X_s-Y_s)ds+\int_0^t \sigma\left(Y_s\right) d B_s+K^Y_t, \quad t \geq 0 \\
\mathbb{E}\left[h\left(Y_t\right)\right] \geq 0, \quad \int_0^t \mathbb{E}\left[h\left(Y_s\right)\right] d K^Y_s=0, \quad t \geq 0
\end{array}\right.\,.
\end{equation}

{From \eqref{eq:Y_in_Q}, we know that under the measure $\mathbb{Q}$, the coefficients of \eqref{eq:6_2_2} satisfy the conditions of Theorem \ref{thm:1}, hence \eqref{eq:Y_in_Q} has a unique solution. Using the Girsanov transformation, we know that \eqref{eq:6_2_2} has a unique solution $\kuo{Y_t}_{t\in [0,T]}$ under $\mathbb{P}$.}

Let 
$$
\tau_n = T \wedge \inf\dkuo{t \in [0,T): |X_t|+|Y_t| \geq n}\,,
$$
we have $\tau_n \rightarrow T$ as $n\rightarrow \infty$. We first show 
\begin{equation*}
\begin{aligned}
	  R_t :=\exp\dkuo{\int_0^t \frac{1}{\xi_s}\sigma\kuo{X_s}^{-1}(Y_s-X_s)dB_s - \frac{1}{2}\int_0^t \frac{\abs{\sigma\kuo{X_s}^{-1}(Y_s-X_s)}^2}{\xi_s^2}ds}\,,
\end{aligned}
\end{equation*}
is a uniformly integrable martingale for $t \in [0,T]$.

\begin{theorem}[Log-Harnack inequality]\label{thm:lhi}
Suppose Assumptions \ref{ass:A_1} and \ref{ass:A_2} hold and $\sigma(x)^{-1}<\infty$. Let $0\leq s<t\leq T$
\begin{enumerate}[(i)]
	\item For any $\mu_0,\nu_0\in \mathcal{P}_2(\mathbb{R})$,
	$$
	\left(P_{s, t} \log f\right)\left(\nu_0\right) \leq \log \left(P_{s, t} f\right)\left(\mu_0\right)+\phi(s, t) \mathcal{W}_2\left(\mu_0, \nu_0\right)^2\,, \quad f \in \mathcal{B}_b^{+}\left(\mathbb{R}\right)\,.
	$$
	Consequently,
	$$
\left|\nabla P_{s, t} f\right|^2 \leq 2 \phi(s, t)\left\{P_{s, t} f^2-\left(P_{s, t} f\right)^2\right\}, \quad f \in \mathcal{B}_b\left(\mathbb{R}\right)\,.
$$
	\item For any different $\mu_0,\nu_0\in \mathcal{P}_2(\mathbb{R})$ and $f \in \mathcal{B}_b^{+}\left(\mathbb{R}\right)$,
	$$
	\begin{aligned}
& \frac{\left|\left(P_{s, t} f\right)\left(\mu_0\right)-\left(P_{s, t} f\right)\left(\nu_0\right)\right|^2}{\mathcal{W}_2\left(\mu_0, \nu_0\right)^2}  \quad \leq 2 \phi(s, t) \sup _{\nu \in B\left(\mu_0, \mathcal{W}_2\left(\mu_0, \nu_0\right)\right)}\left\{\left(P_{s, t} f^2\right)(\nu)-\left(P_{s, t} f\right)^2(\nu)\right\} \,,
\end{aligned}
	$$
 where $B(\mu_0, \mathcal{W}_2(\mu_0, \nu_0))$ is a ball in the Wasserstein spaces $\mathcal{P}_2(\mathbb{R})$ with radius  $\mathcal{W}_2(\mu_0, \nu_0)$.
		Consequently,
		$$
		\begin{aligned}
\left\|P_{s, t}^* \mu_0-P_{s, t}^* v_0\right\|_{v a r} & :=2 \sup _{A \in \mathcal{B}_b\left(\mathbb{R}\right)}\left|\left(P_{s, t}^* \mu_0\right)(A)-\left(P_{s, t}^* v_0\right)(A)\right| \leq \sqrt{2 \phi(s, t)} \mathcal{W}_2\left(\mu_0, v_0\right) .
\end{aligned}
		$$
\end{enumerate}
\end{theorem}
Then we establish Shift Harnack inequality for $P_t$ introduced in \cite{Wang2012IntegrationBP}. We should assume that $\sigma(x,\mu)$ does not depend on $x$. Then, MR-MVSDE \eqref{eq:1.1} becomes 
\begin{equation*}
\left\{\begin{array}{l}
X_t=\xi+\int_0^t b\left(X_s, \mathcal{L}_{X_s}\right) d s+\int_0^t \sigma\left(\mathcal{L}_{X_s}\right) d B_s+K_t, \quad t \geq 0 \\
\mathbb{E}\left[h\left(X_t\right)\right] \geq 0, \quad \int_0^t \mathbb{E}\left[h\left(X_s\right)\right] d K_s=0, \quad t \geq 0
\end{array}\right.\,.
\end{equation*}

\begin{theorem}[Shift Harnack inequality]\label{thm:shi}
 Suppose Assumptions \ref{ass:A_1} and \ref{ass:A_2} hold and $\sigma(\mu)$ is invertible with $\sigma(\mu) + \sigma(\mu)^{-1} $ is bounded.
For any $p>1$, $t\in [0,T]$, $\mu \in \mathcal{P}_2(\mathbb{R})$, $\nu \in \mathcal{P}_2(\mathbb{R})$ and $f\in \mathcal{B}^+_b(\mathbb{R})$, we have 
 	$$
 	\kuo{P_t f}^p(\mu_0)\leq \kuo{P_t f(v + \cdot)^p}(\mu_0)\exp{\frac{p\int_0^t \sigma(\mu_s)^{-2}\dkuo{|v|/t + s|v|/t}^2ds}{2(p-1)}}\,.
 	$$
 	Moreover, for any $f\in \mathcal{B}_b^+(\mathbb{R})$ with $f>1$,
 	$$
 	\kuo{P_t \log f}(\mu_0)\leq \log\kuo{P_t f (v + \cdot)}(\mu_0) + \frac{1}{2}\int_0^t\sigma(\mu_s)^{-2} \dkuo{|v|/t + s|v|/t}^2ds\,.
 	$$
\end{theorem}

\subsection{Proof of Theorem \ref{thm:lhi} and \ref{thm:shi}}
In order to prove Theorem \ref{thm:lhi}, we need the following Lemma.
\begin{lemma}\label{lem:6.3}
	Suppose Assumptions \ref{ass:A_1} and \ref{ass:A_2} hold, and $\sigma(x)^{-1}<\infty$. Let $X_0$, $Y_0$ be two $\mathcal{F}_0-$measurable random variables such that $\mathcal{L}_{\xi^{X_0}} = \mu_0$, $\mathcal{L}_{\xi^{Y_0}} = \nu_0$ and
\begin{equation*}
\begin{aligned}
	  \e\zkuo{\abs{\xi^{X_0}-\xi^{Y_0}}^2} = \mathcal{W}_2\kuo{\mu_0,\nu_0}^2\,.
\end{aligned}
\end{equation*}
Then $\kuo{R_t}_{t\in [0,T]}$ is a uniformly intgrable martingale with 
\begin{equation*}
\begin{aligned}
	  \sup_{t\in [0,T]}\e\zkuo{R_t\log R_t}\leq \phi(0,T)\mathcal{W}_2\kuo{\mu_0,\nu_0}^2\,.
\end{aligned}
\end{equation*}

\end{lemma}
\begin{proof}
By Assumption \ref{ass:A_1}, for any $n\geq 1$, $\kuo{R_{t\wedge\tau_n}}_{t\in [0,T]}$ is a uniformly integrable continuous martingale. Since $\tau_n \rightarrow T$ as $n\rightarrow \infty$, by the martingale convergence Theorem, we only need to prove 
\begin{equation}\label{eq:lem_6_3_1}
\begin{aligned}
	  \sup_{t\in[0,T), n\geq 1}\e\zkuo{R_{t\wedge \tau_n}\log R_{t\wedge \tau_n}}\leq \phi(0,T) \mathcal{W}_2\kuo{\mu_0,\nu_0}^2\,.
\end{aligned}
\end{equation}

We fix $t \in [0,T)$ and $n\geq 1$. By Girsanov's Theorem, 
$$
\tilde{B}_s:=B_s - \frac{1}{\xi_s}\sigma(X_s)^{-1}(Y_s-X_s), s \in [0,t \wedge \tau_n ] \,,
$$
is a Brownian motion under the weighted probability $\mathbb{Q}_{t,n}: = R_{t\wedge \tau_n} \mathbb{P}$.

Under  $\mathbb{Q}_{t,n}$, reformulating \eqref{eq:6_2_1} and \eqref{eq:6_2_2}
\begin{equation*}
\left\{\begin{array}{l}
X_{t\wedge \tau_n}=\xi^{X_0}+\int_0^{t\wedge \tau_n} b\left(X_s, \mathcal{L}_{X_s}\right)-\frac{X_s-Y_s}{\xi_s} d s+\int_0^{t\wedge \tau_n}  \sigma\left(X_s\right) d \tilde{B}_s+\tilde{K}^X_{t\wedge \tau_n} \\
\mathbb{E}^{\mathbb{Q}_{t,n}}\left[h\left(X_t\right)\right] \geq 0, \quad \int_0^t \mathbb{E}^{\mathbb{Q}_{t,n}}\left[h\left(X_s\right)\right] d\tilde{K}^X_s=0
\end{array}\right.\,.
\end{equation*}

\begin{equation*}
\left\{\begin{array}{l}
Y_{t\wedge \tau_n}=\xi^{Y_0}+\int_0^{t\wedge \tau_n} b\left(Y_s, \mathcal{L}_{Y_s}\right)d s+\int_0^{t\wedge \tau_n}  \sigma\left(Y_s\right) d \tilde{B}_s+\tilde{K}^Y_{t\wedge \tau_n} \\
\mathbb{E}\left[h\left(Y_{t\wedge \tau_n}\right)\right] \geq 0, \quad \int_0^{t\wedge \tau_n} \mathbb{E}^{\mathbb{Q}_{t,n}}\left[h\left(Y_s\right)\right] d\tilde{K}^Y_s=0
\end{array}\right.\,.
\end{equation*}
Apply It\^{o}'s Lemma, we have 
\begin{equation*}
\begin{aligned}
	  \kuo{X_{t\wedge\tau_n}-Y_{t\wedge \tau_n}}^2&\leq \int_0^{t\wedge\tau_n}C_1|X_s-Y_s|^2+C_2\mathcal{W}_2\kuo{\mu_s,\nu_s} - \frac{2|X_s-Y_s|^2}{\xi_s}ds\\
	  &+ \int_0^{t\wedge \tau_n}\kuo{X_s-Y_s}\kuo{\sigma(X_s)-\sigma(Y_s)}d\tilde{B}_s + \int_0^{t \wedge \tau_n}(X_s-Y_s)d\kuo{\tilde{K}^X_s - \tilde{K}^Y_s}
\end{aligned}
\end{equation*}
Using the same argument of proving Lemma \ref{lem:6.1}, we have 
$$
\int_0^{t \wedge \tau_n}(X_s-Y_s)d\kuo{\tilde{K}^X_s - \tilde{K}^Y_s}\leq 0\,.
$$
Then, applying Young's inequality, we have 
\begin{equation*}
\begin{aligned}
	  \frac{\kuo{X_{t \wedge \tau_n}-Y_{t\wedge \tau_n}}^2}{\xi_{t \wedge \tau_n}}&\leq \frac{C^2}{2}\int_0^{t \wedge \tau_n}\mathcal{W}_2\kuo{\mu_s,\nu_s}^2ds - \int_0^{t \wedge \tau_n}\kuo{\frac{3}{2} - C_1 \xi_s}\frac{|X_s - Y_s|^2}{\xi_s^2}ds \\
	  &+ \int_0^{t \wedge \tau_n}\frac{\kuo{X_s-Y_s}}{\xi_s}\kuo{\sigma(X_s)-\sigma(Y_s)}d\tilde{B}_s\,.
\end{aligned}
\end{equation*}
By the exponential convergence (Lemma \ref{lem:6.1}), we have 
$$
\mathcal{W}_2\kuo{\mu_s,\nu_s}\leq \mathcal{W}_2\kuo{\mu_0,\nu_0}e^{(C_1+C_2)s}\,.
$$
Then, we have 
\begin{equation*}
\begin{aligned}
	  \e^{\mathbb{Q}_{t,n}}\int_0^{t\wedge \tau_n}\frac{|X_s - Y_s|^2}{\xi_s^2}ds &\leq \frac{2}{\xi_0}+ \e^{\mathbb{Q}_{t,n}}\int_0^{t\wedge \tau_n}C_2^2\mathcal{W}_2\kuo{\mu_s,\nu_s}^2ds \\
	  &\leq \frac{2}{\xi_0}+TC_2^2e^{2(C_1+C_2)T}\mathcal{W}_2\kuo{\mu_0,\nu_0}^2\,.
\end{aligned}
\end{equation*}
Therefore, \eqref{eq:lem_6_3_1} holds since $t \in [0,T)$ and $n\geq 1$ are arbitrary.
\end{proof}
Now, we can prove Theorem \ref{thm:lhi}. 
\begin{proof}[Proof of Theorem \ref{thm:lhi}]
	By Lemma \ref{lem:6.3} and Girsanov's Theorem, $d\mathbb{Q} := R_T d\mathbb{P}$ is a probability measure such that 
	$$
	\tilde{B}_t : = B_t - \int_0^t\frac{\sigma(X_s)^{-1}(Y_s-X_s)}{\xi_s}ds \quad t\in [0,T]\,,
	$$
	is a Brownian motion. Then, \eqref{eq:6_2_2}
 can be written to 
 \begin{equation}\label{eq:Y_in_Q}
\left\{\begin{array}{l}
Y_{t}=\xi^{Y_0}+\int_0^{t} b\left(Y_s, \mathcal{L}_{Y_s}\right)d s+\int_0^{t}  \sigma\left(Y_s\right) d \tilde{B}_s+\tilde{K}^Y_t \\
\mathbb{E}^{\mathbb{Q}}\left[h\left(Y_t\right)\right] \geq 0, \quad \int_0^t \mathbb{E}^{\mathbb{Q}}\left[h\left(Y_s\right)\right] d\tilde{K}^Y_s=0
\end{array}\right.\,.
\end{equation}
Consider MR-MVSDE under $\mathbb{Q}$ measure
 \begin{equation}
\left\{\begin{array}{l}
\tilde{X}_{t}=\xi^{Y_0}+\int_0^{t} b\left(\tilde{X}_s, \mathcal{L}_{\tilde{X}_s}\right)d s+\int_0^{t}  \sigma\left(\tilde{X}_s\right) d \tilde{B}_s+\tilde{K}^{\tilde{X}}_t \\
\mathbb{E}^{\mathbb{Q}}\left[h\left(\tilde{X}_t\right)\right] \geq 0, \quad \int_0^t \mathbb{E}^{\mathbb{Q}}\left[h\left(\tilde{X}_s\right)\right] d\tilde{K}^Y_s=0
\end{array}\right.\,.
\end{equation}
By Theorem \ref{thm:1}, we have $\tilde{X}_t = Y_t$ for $t\in [0,T]$. So we deduce that 
$$
\kuo{P_T \log f}(\nu_0) \leq \log\kuo{P_T f}(\mu_0) + \phi(0,T)\mathcal{W}_2\kuo{\mu_0,\nu_0}^2\,.
$$
In particular,
$$
P_T\log f(x)\leq \kuo{\log P_T f}(y) + \phi(0,T) |x-y|^2\,.
$$
Following \cite{arnaudon2014equivalent}, we can get for any $f\in \mathcal{B}_b^+(\mathbb{R})$
$$
\left|\nabla P_{s,t} f\right|^2\leq 2\phi(s,t)\zkuo{P_{s,t}f^2 - \kuo{P_{s,t}f}^2}\,.
$$
Repeating the proof of \cite{wang2018distribution}[Theorem 4.1 (ii)], we can obtain for any different $\mu_0$ and $\nu_0\in \mathcal{P}_2(\mathbb{R})$ and any $f\in \mathcal{B}_b^+(\mathbb{R})$,
\begin{equation*}
\begin{aligned}
& \frac{\left|\left(P_{s, t} f\right)\left(\mu_0\right)-\left(P_{s, t} f\right)\left(\nu_0\right)\right|^2}{\mathcal{W}_2\left(\mu_0, \nu_0\right)^2} \\
& \quad \leq 2 \phi(s, t) \sup _{\nu \in B\left(\mu_0, \mathcal{W}_2\left(\mu_0, \nu_0\right)\right)}\left\{\left(P_{s, t} f^2\right)(\nu)-\left(P_{s, t} f\right)^2(\nu)\right\} \,,
\end{aligned}
\end{equation*}
and 
\begin{equation*}
\begin{aligned}
\left\|P_{s, t}^* \mu_0-P_{s, t}^* v_0\right\|_{v a r} & :=2 \sup _{A \in \mathcal{B}_b\left(\mathbb{R}^d\right)}\left|\left(P_{s, t}^* \mu_0\right)(A)-\left(P_{s, t}^* v_0\right)(A)\right| \\
& \leq \sqrt{2 \phi(s, t)} \mathcal{W}_2\left(\mu_0, v_0\right)\,.
\end{aligned}
\end{equation*}
Then, we complete this proof.
 \end{proof}
	

\begin{proof}[Proof of Theorem \ref{thm:shi}]
	Let $Y_t = X_t + \frac{tv}{T}$, then 
	$$
	dY_t = b(Y_t,\mu_t)dt + \sigma(\mu_t)d\tilde{B}_t + dK_t\,,
	$$ 
	where 
	$$
	\begin{aligned}
		\tilde{B}_t &: = B_t + \int_0^t \eta_s ds\,,\\
		\eta_t &: = \sigma(\mu_t)^{-1}\dkuo{\frac{v}{T}+ b(X_t,\mu_t) - b(X_t + \frac{t v}{T},\mu_t)}.
	\end{aligned}
	$$
	Let $R_T = \exp\dkuo{\int_0^T \eta_s dW_s - \frac{1}{2}\int_0^T |\eta_s|^2ds}$. Since $\e\zkuo{\exp(\frac{1}{2}\int_0^T |\eta_s|^2ds)}< \infty$, from Nokikov's condition, we deduce that $R_t$ is a martingale and $\e R_T = 1$. By the Girsanov's Theorem, we can get
	$$
	\kuo{P_Tf}(\mu_0) = \e\zkuo{R_T f(Y_T)} = \e\zkuo{R_Tf(X_T + v)}\leq \kuo{P_T f(v+\cdot)^p}^{\frac{1}{p}}(\mu_0) \kuo{\e R_T^{\frac{p}{p-1}}}^{\frac{p-1}{p}}\,.
	$$
	Applying Young's inequality, we obtain 
	$$
	\begin{aligned}
			\kuo{P_T\log f}(\mu_0) &= \e\zkuo{R_T \log f(Y_T)} = \e\zkuo{R_T\log f\kuo{X_T + v}}\\
			&\leq \log P_T f(v + \cdot)(\mu_0) + \e\zkuo{R_T\log R_T}\,.
	\end{aligned}
	$$
	Then we have 
$$
	  \e R_T^{\frac{p}{p-1}}\leq \exp\dkuo{\frac{p}{2(p-1)^2}\int_0^T |\eta_s|^2ds}\leq \exp\dkuo{\frac{p \int_0^T \sigma(\mu_t)^{-2}\dkuo{|v|/T + t|v|/T}^2dt}{2(p-1)^2}}\,,
$$
and 
$$
\e \zkuo{ R_T \log R_T} = \e^{\mathbb{Q}}\log R_T \leq \frac{1}{2}  \e^{\mathbb{Q}} \int_0^T|\eta_s|^2ds \leq \frac{1}{2}\int_0^T \sigma(\mu_t)^{-2}\dkuo{|v|/T + t|v|/T}^2dt\,.
$$
\end{proof}
\section*{Appendix}
\subsection*{The proof of Lemma \ref{lem:2.1}.}
\begin{proof}[Proof of Lemma \ref{lem:2.1}.]
(i)
	By the definition of $H(x,\nu)$ and the bi-Lipschitz assumption of $h$ we have for any $\nu \in \mathcal{P}_1(\mathbb{R})$
	$$
	\begin{aligned}
		\abs{H(x,\nu)-H(x^\prime,\nu)} &= \abs{\int h(x+z)-h(x^\prime+z)d\nu(z) }\\
		&\leq M \abs{\int x-x^\prime d\nu(z)}\leq C |x-x^\prime|,
	\end{aligned}
	$$
	and 
	$$
	\begin{aligned}
		\abs{H(x,\nu)-H(x^\prime,\nu)} &= \abs{\int h(x+z)-h(x^\prime+z)d\nu(z) }\\
		&\geq m\abs{\int x-x^\prime d\nu(z)}\geq c|x-x^\prime|.
	\end{aligned}
	$$

	(ii) For any $\nu$, $\nu^\prime \in \mathcal{P}_1(\mathbb{R})$, by definetion we have 
	$$
	\begin{aligned}
		\abs{H(x,\nu)-H(x,\nu^\prime)}\leq \abs{\int h(x+\cdot )(d\nu-d\nu^\prime)}.
	\end{aligned}
	$$
\end{proof}

\subsection*{Well-posedness of \eqref{eq:s_DP_1}}
\begin{theorem}\label{thm:well_posedness_sDP1}
	Under Assumptions \ref{ass:A_1} and \ref{ass:A_2}, the MR-MVSDE \eqref{eq:s_DP_1} has a unique deterministic flat solution $(X,K)$. Moreover, 
\begin{equation}
\forall t \geq 0, \quad K_t=\sup _{s \leq t} \inf \left\{x \geq 0: \mathbb{E}\left[h\left(x+U_s\right)\right] \geq 0\right\}=\sup _{s \leq t} G_0\left(\law{U_s}\right),
\end{equation}
where $\kuo{U_t}_{0\leq t\leq T}$ is the process defined by :
\begin{equation}\label{eq:U}
\begin{aligned}
	  U_t = \xi + \int_0^t b(X_s,\mathcal{L}_{X_s})dA_s + \int_0^t \sigma(X_s,\mathcal{L}_{X_s})dM_s.
\end{aligned}
\end{equation}
With these notations, denoting by $\law{U_s}$ the family of marginal laws of $\kuo{U_t}_{0\leq t\leq T}$ we have 
$$
K_t = \sup_{s\leq t}G_0(\law{U_s}).
$$
\end{theorem}
\begin{proof}
	Let us consider the set $\mathcal{C}^2 = \dkuo{X \text{ is } \mathcal{F}-\text{adapted continuous process}, \e\kuo{\sup_{0\leq t\leq T}|X_t|^2< \infty}} $. And let $\widetilde{X}\in \mathcal{C}^2$ be a given process. We set 
\begin{equation}
\tilde{U}_t=\xi+\int_0^t b\left(\tilde{X}_s, \mathcal{L}_{\tilde{X}_s}\right) d A_s+\int_0^t \sigma\left(\tilde{X}_s, \mathcal{L}_{\tilde{X}_s}\right) d M_s, \quad \operatorname{Law}\left(\tilde{U}_s\right)=: \law{\tilde {U}_s}\,,
\end{equation}
and define the function $K$ by setting
$$
{K_t=\sup _{s \leq t} \inf \left\{x \geq 0: \mathbb{E}\left[h\left(x+\tilde{U}_s\right)\right] \geq 0\right\}=G_0\left(\law{\tilde {U}_s}\right) .}
$$
The function $K$ being given, let us define the process $X$ by the formula
\begin{equation}
\begin{aligned}
	  X_t=\xi+\int_0^t b\left(\tilde{X}_s,\mathcal{L}_{\tilde{X}_s}\right) d A_s+\int_0^t \sigma\left(\tilde{X}_s,\mathcal{L}_{\tilde{X}_s}\right) d M_s+K_t .
\end{aligned}
\end{equation}

Based on the definition of $K$, we have $\e\zkuo{h(X_t)}\geq 0$, $K_t = G_0\kuo{\law{\tilde{U}_t}}, dK_t-a.e.$ and $G_0(\law{\tilde{U}_t})>0, dK_t-a.e.$ Then we obtain 
$$
\begin{aligned}
\int_0^t \mathbb{E}\left[h\left(X_s\right)\right] d K_s & =\int_0^t \mathbb{E}\left[h\left(\tilde{U}_s+K_s\right)\right] d K_s \\
& =\int_0^t \mathbb{E}\left[h\left(\tilde{U}_s+G_0\left(\law{\tilde{U}_s}\right)\right)\right] d K_s \\
& =\int_0^t \mathbb{E}\left[h\left(\tilde{U}_s+G_0\left(\law{\tilde{U}_s}\right)\right)\right] \mathbbm{1}_{\dkuo{G_0\left(\law{\tilde{U}_s}\right)>0}} d K_s .
\end{aligned}
$$
Moreover, since $h$ is continuous, we have $\mathbb{E}\left[h\left(\tilde{U}_s+G_0\left(\law{\tilde{U}_s}\right)\right)\right]=0$ as soon as $G_0\left(\law{\tilde{U}_s}\right)>0$, so that
$$
\int_0^t \mathbb{E}\left[h\left(X_s\right)\right] d K_s=0\,.
$$

 Let $\tilde{X}, \tilde{X}^\prime \in \mathcal{C}^2$ and define $\tilde{U}$, $K$ and $\tilde{U}^\prime$, $K^\prime$ as above, using the some Brownian motion. We have from the Lipschitz assumption of $b$ and $\sigma$, B-D-G inequality 
 \begin{equation}\label{eq:A_Th_1.1}
\begin{aligned}
	  &\e\zkuo{\sup_{t\leq T}|\tilde{X}_t - \tilde{X}_t^\prime|^2}\\
	  &\leq  3\e \left[\sup_{t\leq T}\left\{\abs{\int_0^t\kuo{b\kuo{\tilde{X}_s,\mathcal{L}_{\tilde{X}_s}}- b\kuo{\tilde{X}^\prime_s,\mathcal{L}_{\tilde{X}^\prime_s}} 
	  }dA_s}^2 + \abs{\int_0^t \kuo{\sigma\kuo{\tilde{X}_s,\mathcal{L}_{\tilde{X}_s}}-\sigma\kuo{\tilde{X}^\prime_s,\mathcal{L}_{\tilde{X}^\prime_s}}}dM_s}^2\right.\right.\\
	  &\left.\left.\qquad +\abs{K_t - K^\prime_t}^2 \right\}\right]\\
	  &\leq3 \mathbb{E}\left[\left.T\cdot \sup _{t \leq T} \int_0^t\left(b\left(\tilde{X}_s, \mathcal{L}_{\tilde{X}_s}\right)-b\left(\tilde{X}_s^{\prime}, \mathcal{L}_{\tilde{X}_s^{\prime}}\right)\right) d A_s\right|^2\right]\\
     &\qquad+3 \mathbb{E}\left[\sup _{t \leq T}\left|\int_0^t\left(\sigma\left(\tilde{X}_s, \mathcal{L}_{\tilde{X}_s}\right)-\sigma\left(\tilde{X}^{\prime}_s, \mathcal{L}_{\tilde{X}^{\prime}_s}\right)\right) d M_s\right|^2\right]+ 3 \sup_{t\leq T}\abs{K_t - K^\prime_t}^2 \\
	  & \leq C\int_0^t\e\zkuo{\sup_{t\leq T}|\tilde X_t-\tilde X^\prime_t|^2}d(A_s+\brackt{M}_s) + 3 \sup_{t\leq T}\abs{K_t - K^\prime_t}^2   \,.
\end{aligned}
\end{equation}
From the representation of $K$ and Lemma \ref{lem:2.2}, we have 
\begin{equation}\label{eq:A_K_t}
\begin{aligned}
	  \sup_{t\leq T}|K_t-K^\prime_t|^2 &= \sup_{t\leq T}\abs{\sup_{s\leq t}G_0(\law{\tilde X_t})-\sup_{s\leq t}G_0(\law{\tilde X_t^\prime})}\leq \sup_{t\leq T}\abs{G_0(\law{\tilde X_t})-G_0(\law{\tilde X_t^\prime})}\\
	  &\leq \frac{M}{m} \mathbb{E}\left[\sup _{t \leq T}\left|\tilde{U}_t-\tilde{U}^{\prime}_t\right|^2\right] \leq  C\mathbb{E}\left[\sup _{t \leq T}\left|\tilde{X}_t-\tilde{X}_t^{\prime}\right|^2\right] \,.
\end{aligned}
\end{equation}
Plugging \eqref{eq:A_K_t} into \eqref{eq:A_Th_1.1}, and use the stochastic Gronwall lemma, we have 
$$
\e\zkuo{\sup_{t\leq T}|\tilde X_t - \tilde X_t^\prime|^2}\leq  C(1+T) T \mathbb{E}\left[\sup _{t \leq T}\left|\tilde{X}_t-\tilde{X}_t^{\prime}\right|^2\right].
$$

Hence, there exists a positive $\tilde{T}$, depending on $b, \sigma$ and $h$ only, such that for all $T<\tilde{T}$, the map $\Phi$ is a contraction. We first deduce the existence and uniqueness of the solution on $[0, \tilde{T}]$ and then on $\mathbb{R}^{+}$ by iterating the construction.

\end{proof}

\section*{Acknowledgements}
The authors would like to thank the referees for their very constructive suggestions and valuable comments. 

 \bibliographystyle{elsarticle-num} 
\bibliography{cas-refs}

\end{document}